\documentclass[twoside, a4paper, 10pt]{amsart}

\usepackage{amsmath}  
\usepackage{amsfonts}
\usepackage{amsthm}  
\usepackage{amssymb}
\usepackage[alphabetic, abbrev]{amsrefs}

\usepackage{enumerate}
\usepackage[all]{xy}
\SelectTips{cm}{}
\SilentMatrices
\usepackage{hyperref}

\numberwithin{equation}{section}
\theoremstyle{plain}
\newtheorem{lemma}[subsection]{Lemma}
\newtheorem{theorem}[subsection]{Theorem}

\newtheorem{corollary}[subsection]{Corollary}
\newtheorem{proposition}[subsection]{Proposition}
\theoremstyle{definition}

\newtheorem{definition}[subsection]{Definition}
\newtheorem{example}[subsection]{Example}
\newtheorem{remark}[subsection]{Remark}

\newcommand{\mC}{{\mathbb C}}

\newcommand{\mN}{{\mathbb N}}

\newcommand{\mS}{{\mathbb S}}
\newcommand{\mZ}{{\mathbb Z}}

\newcommand{\cC}{{\mathcal C}}

\newcommand{\cI}{{\mathcal I}}
\newcommand{\cJ}{{\mathcal J}}

\newcommand{\cK}{{\mathcal K}}
\newcommand{\cL}{{\mathcal L}}
\newcommand{\cM}{{\mathcal M}}

\newcommand{\cO}{{\mathcal O}}
\newcommand{\cP}{{\mathcal P}}
\newcommand{\cS}{{\mathcal S}}

\newcommand{\cW}{{\mathcal W}}

\DeclareMathOperator{\id}{id}

\DeclareMathOperator{\Map}{Map}

\DeclareMathOperator{\Ho}{Ho}

\DeclareMathOperator{\colim}{colim}
\DeclareMathOperator{\hocolim}{hocolim}
\DeclareMathOperator{\obj}{Ob}

\DeclareMathOperator{\THH}{THH}

\DeclareMathOperator{\concat}{\sqcup}

\newcommand{\eins}{\mathbf{1}}
\newcommand{\tensor}{\otimes}
\newcommand{\ovl}{\overline}

\newcommand{\ot}{\leftarrow}
\newcommand{\iso}{\cong}

\newcommand{\op}{{\mathrm{op}}}
\newcommand{\sm}{\wedge}

\newcommand{\wdg}{\vee}

\newcommand{\Spsym}{{\mathrm{\textrm{Sp}^{\Sigma}}}}
\newcommand{\SpN}{{\mathrm{\textrm{Sp}^{\mN}}}}
\newcommand{\bld}[1]{{\mathbf{#1}}}

\DeclareMathOperator{\capitalGL}{GL}
\DeclareMathOperator{\gl}{gl}
\DeclareMathOperator{\bgl}{bgl}

\newcommand{\bof}[1]{b{#1}}

\newcommand{\OmegaI}{\Omega^{\cI}}
\newcommand{\OmegaIof}[1]{\Omega^{\cI}\!{#1}}
\newcommand{\OmegaJ}{\Omega^{\cJ}}
\newcommand{\OmegaJof}[1]{\Omega^{\cJ}\!{#1}}

\newcommand{\gloneIof}[1]{\gl^{\cI}_1\!\!{#1}}
\newcommand{\gloneI}{\gl^{\cI}_1}
\newcommand{\gloneJof}[1]{\gl^{\cJ}_1\!\!{#1}}
\newcommand{\gloneJ}{\gl^{\cJ}_1}
\newcommand{\glonegraded}[1]{\gl^{*}_1\!\!{#1}}
\newcommand{\bglonegraded}{\bgl^{*}_1}
\newcommand{\bglonegradedof}[1]{\bgl^{*}_1\!\!{#1}}
\newcommand{\bglone}{\bgl_1}
\newcommand{\bgloneof}[1]{\bgl_1\!\!{#1}}

\newcommand{\gloneof}[1]{\gl_1\!\!{#1}}

\newcommand{\GLoneIof}[1]{\capitalGL^{\cI}_1\!{#1}}
\newcommand{\GLoneJof}[1]{\capitalGL^{\cJ}_1\!\!{#1}}

\newcommand{\GammaJ}{\Gamma^{\op}\!\cJ}
\newcommand{\GammaJS}{\Gamma^{\op}\!\cJ\!\text{-}\cS}
\newcommand{\GammaS}{\Gamma^{\op}\!\text{-}\cS}

\newcommand{\HcJ}{H\!\cJ}
\newcommand{\HcK}{H\!\cK}

\newcommand{\lev}{\mathrm{lev}}
\newcommand{\pre}{\mathrm{pre}}
\newcommand{\sta}{\mathrm{st}}
\newcommand{\abs}{\mathrm{abs}}
\newcommand{\pos}{\mathrm{pos}}
\newcommand{\grp}{\mathrm{gp}}
\newcommand{\cof}{\mathrm{cof}}
\newcommand{\fib}{\mathrm{fib}}

\DeclareMathOperator{\diag}{diag}

\newcommand{\arxivlink}[1]{\href{http://arxiv.org/abs/#1}{\texttt{arXiv:#1}}}
\newcommand{\doilink}[1]{\href{http://dx.doi.org/#1}{doi:#1}}

\begin{document} 
\title[Spectra of units and group completion of graded \texorpdfstring{$E_{\infty}$}{E-infinity} spaces]{Spectra of units for periodic ring spectra and group completion of graded \texorpdfstring{$E_{\infty}$}{E-infinity} spaces}

\author{Steffen Sagave} \address{Mathematical
Institute, University of Bonn, Endenicher Allee 60, 53115 Bonn,
Germany} \email{sagave@math.uni-bonn.de}

\date{\today}
\begin{abstract}
  We construct a new spectrum of units for a commutative symmetric
  ring spectrum that detects the difference between a periodic ring
  spectrum and its connective cover. It is augmented over the sphere
  spectrum. The homotopy cofiber of its augmentation map is a
  non-connected delooping of the usual spectrum of units whose bottom
  homotopy group detects periodicity.

  Our approach builds on the graded variant of
  \texorpdfstring{$E_{\infty}$}{E-infinity} spaces introduced in joint
  work with Christian Schlichtkrull. We construct a group completion
  model structure for graded \texorpdfstring{$E_{\infty}$}{E-infinity}
  spaces and use it to exhibit our spectrum of units functor as right
  adjoint on the level of homotopy categories. The resulting group
  completion functor is an essential tool for studying ring spectra
  with graded logarithmic structures.
\end{abstract}
\maketitle

\section{Introduction} 
Having a notion of units for ring spectra is useful in connection with
algebraic \mbox{$K$-theory} and Thom spectra. Sufficiently commutative
ring spectra give rise to spectra of units. This was first made
precise by May, Quinn, Ray, and Tornehave~\cite{May_ring_spaces} who
defined a spectrum of units associated with an $E_{\infty}$ ring
spectrum and showed how it controls its orientations.  Ando, Blumberg,
Gepner, Hopkins, and Rezk~\cite{ABGHR_Thom-infinity,ABGHR_Thom-rigid} 
generalized this by giving an extensive treatment of units, Thom
spectra, and orientation theory in the context of structured ring
spectra and  $\infty$-categories.

In joint work with Schlichtkrull~\cite{Sagave-S_diagram}, we
introduced the \emph{graded units} $\GLoneJof{E}$ of a commutative
symmetric ring spectrum $E$. One key feature of these graded units is
that they detect the difference between a periodic ring spectrum and
its connective cover, which is essential for applications to
logarithmic ring spectra (that are outlined at the end of this
introduction).  The object $\GLoneJof{E}$ can be viewed as a space
valued lax symmetric monoidal on a certain indexing category
$\cJ$. One aim of the present paper is to explain how this approach to
graded units leads to \emph{spectra} of graded units that detect
periodicity. To this end, we construct a functor
\begin{equation}\label{eq:GLonehJ-intro}
\gloneJ\colon \cC\Spsym \to \GammaS / \bof{\cJ},\qquad E \mapsto \left(\gloneJof{E} \to \bof{\cJ}\right).
\end{equation}
Here $\cC\Spsym$ is the category of commutative symmetric ring
spectra, $\GammaS$ is the category of $\Gamma$-spaces in the sense of
Segal, and $\bof{\cJ}$ is a certain explicit $\Gamma$-space coming
from a symmetric monoidal category. The spectrum associated with
$\bof{\cJ}$ is an $\Omega$-spectrum that is stably equivalent to the
sphere spectrum, i.e., an $\Omega$-spectrum with zeroth space $QS^0
\simeq \Omega^{\infty}\Sigma^{\infty}S^0$. Since $\Gamma$-spaces model
connective spectra, the category $\GammaS/\bof{\cJ}$ is a convenient
model for the category of connective spectra over the sphere
spectrum. We explain below why it is reasonable for a spectrum of
units to take values in this category, rather than just in connective
or non-connective spectra.

Commutative symmetric ring spectra are strictly commutative models for
$E_{\infty}$ spectra, and the definition of $\gloneJ$ makes explicit
use of their combinatorics. A commutative symmetric ring spectrum $E$
can be defined as a collection of pointed spaces $E_m$ together with a
basepoint preserving left action of the symmetric group $\Sigma_m$,
multiplication maps $E_m \sm E_n \to E_{m+n}$, and unit maps $S^{m}
\to E_m$ that satisfy appropriate relations. By definition, the value
of the $\Gamma$-space $\gloneJof{E}$ at the based set
$k^+=\{0,\dots,k\}$ is a certain homotopy colimit over a $k$-fold
product of subspaces consisting of those path components of $\Omega^{n_2}E_{n_1}$
that correspond to units in the underlying multiplicative graded
monoid of $\pi_*(E)$. The augmentation to $\bof{\cJ} \iso
\gloneJof{(*)}$ results from the choice for the indexing category of
the homotopy colimits.

\subsection{A non-connected delooping of \texorpdfstring{$\gloneof{E}$}{gl\_1E}}
A first indication for why $\gloneJof{E}$ is of interest is the
following relation to ordinary spectra of units. Let
$\bglonegradedof{E}$ be the spectrum associated with the homotopy
cofiber of the augmentation \mbox{$\gloneJof{E}\to \bof{\cJ}$}.
\begin{theorem}\label{thm:bglonegraded-introduction}
  Let $E$ be a positive fibrant commutative symmetric ring spectrum.
  Then $\bglonegradedof{E}$ is a connective spectrum, and the ordinary
  spectrum of units $\gloneof{E}$ is the connective cover of
  $\Omega\bglonegradedof{E}$.

  The bottom homotopy group $\pi_0(\bglonegradedof{E})$ is isomorphic
  to $\mZ/n_E \mZ$ where $n_E\in\mN_0$ is the \emph{periodicity} of
  $E$. By definition, $n_E = 0$ (and $\mZ/n_E \mZ \iso \mZ$) if all
  units of the underlying multiplicative graded monoid of $\pi_*(E)$
  have degree $0$, and $n_E$ is the smallest positive degree of a unit
  in $\pi_*(E)$ otherwise.
\end{theorem}
The theorem says that $\bglonegradedof{E}$ is a not necessarily
connected delooping of the usual spectrum of units whose bottom
homotopy group detects periodicity. If $E$ is for example
$2$-periodic, then its connective cover $e \to E$ induces the
surjection
\[ \mZ \iso \pi_0(\bglonegradedof{e}) \to \pi_0(\bglonegradedof{E}) \iso \mZ/2. \]
By construction there is a map of spectra $\mS  \to \bglonegradedof{E}$. 
The induced map 
\[ \mZ/2 \iso \pi_1(\mS) \to \pi_1(\bglonegradedof{E}) \iso
\pi_0(\gloneof{E}) \iso (\pi_0(E))^{\times} \] turns out to be the
sign action of the additive group structure on $\pi_0(E)$. This
relation of the stable Hopf map that generates $\pi_1(\mS)$ and the
units does not seem to be visible if one only considers the ordinary
$\gloneof{E}$.

A different way of extending the notion of units of $E$ so that
periodicity gets detected is to form the \emph{Picard space} of
$E$. By definition, $\mathrm{Pic}(E)$ is the $\infty$-groupoid of
invertible elements in the symmetric monoidal $\infty$-category of
$E$-modules; see e.g.~\cite[Section 8]{Ando-B-G_parametrized}. The
bottom homotopy group of $\mathrm{Pic}(E)$ is the Picard group of $E$,
and we expect that $\Omega^{\infty}(\bglonegradedof{E})$ is equivalent
to the full $\infty$-subgroupoid of $\mathrm{Pic}(E)$ spanned by the
shifts of $E$.

In view of the $E$-algebra Thom spectra associated with maps to
$\bglone{E}$ or $\mathrm{Pic}(E)$ constructed by Ando et
al.~\cite{ABGHR_Thom-rigid,Ando-B-G_parametrized}, it is natural to
ask if maps to $\bglonegradedof{E}$ give rise to graded $E$-algebra
Thom spectra. The construction of such a graded Thom spectrum functor
will be given in a forthcoming joint project with
Schlichtkrull~\cite{Sagave-S_Thom-graded-units}. There we also give an
alternative model of the $E_{\infty}$ space
$\Omega^{\infty}(\bglonegradedof{E})$ that can be interpreted as a
classifying space for $\GLoneJof{E}$-modules.

\subsection{\texorpdfstring{$\cI$}{I}-spaces and
  \texorpdfstring{$E_{\infty}$}{E-infinity} spaces}\label{subsec:I-and-E_infty-introduction} Our main tools for
building and studying the functor $\gloneJ$ are methods for the
homotopical analysis of diagram spaces developed together with
Schlichtkrull in~\cite{Sagave-S_diagram}. One important instance for
such diagram spaces are $\cI$-spaces. These are space valued functors
on the category of finite sets $\bld{m}=\{1,\dots,m\}, m\geq 0$, and
injective maps. The category of $\cI$-spaces $\cS^{\cI}$ has a
symmetric monoidal product $\boxtimes$ induced by the ordered
concatenation of finite sets and the cartesian product of spaces.

The commutative monoids in $(\cS^{\cI},\boxtimes)$ are called
\emph{commutative $\cI$-space monoids} and form a category denoted by
$\cC\cS^{\cI}$. This category admits a \emph{positive $\cI$-model
  structure} such that $\cC\cS^{\cI}$ is Quillen equivalent to the
category of $E_{\infty}$ spaces. So all $E_{\infty}$-spaces admit a
strictly commutative model in $(\cS^{\cI},\boxtimes)$. The point here
is that the extra symmetry of $\cI$-spaces and the use of a
\emph{positive} model structure ensure that we do not try to represent
the homotopy type of an arbitrary $E_{\infty}$ space by a strictly
commutative monoid in spaces (which would be a contradiction). This is
related to J. Smith's insight that only a positive model structure on
symmetric spectra lifts to commutative symmetric ring
spectra~\cite{MMSS}.

If $E$ is a commutative symmetric ring spectrum, there is a
commutative $\cI$-space monoid $\OmegaIof{E}$ representing its
multiplicative $E_{\infty}$ space. It is on objects defined by
$(\OmegaIof{E})(\bld{m}) = \Omega^{m}E_m$. The grouplike $E_{\infty}$
space of units of $E$ is represented by a sub commutative $\cI$-space
monoid $\GLoneIof{E}$ of $\OmegaIof{E}$. It gives rise to a very
special $\Gamma$-space $\gloneIof{E}$ whose associated spectrum models
the ordinary spectrum of
units~\cite{Schlichtkrull_units,Lind-diagram}.

The passage from $\GLoneIof{E}$ to $\gloneIof{E}$ is an implementation
of the \emph{recognition principle} which states that grouplike
$E_{\infty}$~spaces have the homotopy types of infinite loop
spaces. In~\cite{Sagave-S_group-compl} we prove in joint work with
Schlichtkrull that the category $\cC\cS^{\cI}$ admits a \emph{group
  completion model structure} $\cC\cS^{\cI}_{\grp}$ whose fibrant
objects are positive $\cI$-fibrant and grouplike. We further show that
$\cC\cS^{\cI}_{\grp}$ is Quillen equivalent to $\Gamma$-spaces with a
stable model structure. The resulting equivalence between the homotopy
categories of grouplike commutative $\cI$-space monoids and connective
spectra may be viewed as an incarnation of the recognition
principle. We use this Quillen equivalence
in~\cite{Sagave-S_group-compl} to show that the functor of homotopy
categories $\Ho(\cC\Spsym) \to \Ho(\GammaS)$ induced by $\gloneI$ is a
right adjoint, reproving  a result by Ando et al.~\cite[Theorem
5.1]{ABGHR_Thom-rigid}.

\subsection{\texorpdfstring{$\cJ$}{J}-spaces and graded
  \texorpdfstring{$E_{\infty}$}{E-infinity} spaces} The functor
$\gloneI$ does not detect periodicity because the commutative
$\cI$-space monoids $\OmegaIof{E}$ and $\GLoneIof{E}$ have no
information about the negative dimensional homotopy groups of
$E$. Together with Schlichtkrull we showed in~\cite{Sagave-S_diagram}
that one can overcome this by using a different indexing category
$\cJ$ instead of $\cI$. This $\cJ$ is defined to be Quillen's
localization construction $\Sigma^{-1}\Sigma$ on the category of
finite sets and bijections $\Sigma$. It is symmetric monoidal and
gives rise to $\cJ$-spaces and a category of commutative $\cJ$-space
monoids $\cC\cS^{\cJ}$ just as in the case of $\cI$. The classifying
space $B\cJ$ of $\cJ$ has the homotopy type of
$QS^0$. Objects of $\cJ$ are pairs of finite sets
$(\bld{n_1},\bld{n_2})$, and a close relation between $\cJ$ and the
combinatorics of symmetric spectra ensures that there is a functor
$\OmegaJ\colon \cC\Spsym \to \cC\cS^{\cJ}$ with
$(\OmegaJof{E})(\bld{n_1},\bld{n_2}) = \Omega^{n_2}E_{n_1}$. There is
a sub commutative $\cJ$-space monoid $\GLoneJof{E}$ of $\OmegaJof{E}$
that is used to define the $\Gamma$-space $\gloneJof{E}$ outlined
in~\eqref{eq:GLonehJ-intro} above.

As for $\cI$-spaces, there is a positive $\cJ$-model structure for
$\cC\cS^{\cJ}$. With this model structure, $\cC\cS^{\cJ}$ is related
by a chain of Quillen equivalences to the category of $E_{\infty}$
spaces over $B\cJ$. This result is a reason for why we think of
commutative $\cJ$-space monoids as ``graded $E_{\infty}$ spaces'':
While a $\mZ$-graded monoid in algebra can be defined as a commutative
monoid over the underlying additive monoid $(\mZ,+)$ of the initial
commutative ring $\mZ$, a commutative $\cJ$-space monoid is up to
homotopy the same as an $E_{\infty}$ space over $QS^0$, the underlying
``additive'' infinite loop space of the initial commutative ring
spectrum $\mS$. So $QS^0 \simeq B\cJ$ plays the role of
$(\mZ,+)$. From this point of view, $\OmegaJof{E}$ is the underlying
graded multiplicative $E_{\infty}$ space of $E$, and $\GLoneJof{E}$ is
the ``graded units'' of $E$.

\subsection{Group completion of graded commutative spaces}
Combining the last result and the recognition principle mentioned
above suggests that grouplike commutative $\cJ$-space monoids should
correspond to connective spectra over the sphere spectrum. One of the
main result of the present paper is a proof of this statement on the
level of model categories.
\begin{theorem}\label{thm:gp-model-str-and-comparison-introduction}
  The category of commutative $\cJ$-space monoids admits a group
  completion model structure in which the fibrant objects are the
  grouplike positive $\cJ$-fibrant objects and the fibrant replacement
  is a group completion.  

  There is a chain of Quillen equivalences relating this model
  category to the category of $\Gamma$-spaces over $\bof{\cJ}$ with a
  stable model structure.
\end{theorem}
In the theorem, a map of commutative $\cJ$-space monoids is a group
completion if the associated map of $E_{\infty}$ spaces is a group
completion in the usual sense.  The proof of the theorem is more
difficult than (and different to) the proof of the corresponding
result about $\cI$-spaces in~\cite{Sagave-S_group-compl}: In the
context of $\cJ$-spaces, the group completion functor $\Omega B$ for
simplicial monoids does not lift to $\cC\cS^{\cJ}$. Moreover, we only
get a chain of Quillen equivalences here, rather than a single Quillen
equivalence as in the case of $\cI$-spaces.

Together with a model category treatment of the passage from
$\OmegaJof{E}$ to $\GLoneJof{E}$ that works as in the case of
$\cI$-spaces,
Theorem~\ref{thm:gp-model-str-and-comparison-introduction} is the
key ingredient for
\begin{theorem}\label{thm:sp-of-graded-units-as-righ-adj-introduction}
  The functor $\Ho(\cC\Spsym)\to\Ho(\GammaS/\bof{\cJ})$ induced by
  $\gloneJ$ is a right adjoint.
\end{theorem}
This theorem may be viewed as the graded analog of the result about
ordinary units mentioned in
Section~\ref{subsec:I-and-E_infty-introduction}. Group completions are relevant for this theorem about units because  the stably
fibrant replacement of a special $\Gamma$-space models the group
completion.

\subsection{Applications to ring spectra with logarithmic structures}\label{subsec:log-applications}
A \emph{graded pre-log ring spectrum} is a commutative symmetric ring
spectrum $E$ together with a commutative $\cJ$-space monoid $M$ and a
map $M \to \Omega^{\cJ}(E)$ in $\cC\cS^{\cJ}$. A variant of this
notion employing $\cI$-spaces instead of $\cJ$-spaces was introduced
by Rognes in~\cite{Rognes_TLS} in order to use ideas from logarithmic
algebraic geometry to understand various phenomena related to
topological Hochschild homology and algebraic $K$-theory.  As
explained in~\cite[\S 4.30]{Sagave-S_diagram}, the $\cJ$-space version
of this definition has the advantage that it easier to extend
topological $K$-theory spectra to (graded) pre-log ring spectra in an
interesting way.

The results of the present paper are relevant for the study of pre-log
ring spectra. The group completion functor for commutative $\cJ$-space
monoids resulting from
Theorem~\ref{thm:gp-model-str-and-comparison-introduction} is an
essential foundation for the author's work on the \emph{logarithmic
  topological Andr\'{e}-Quillen
  homology}~\cite{Sagave_log-on-k-theory} and the author's joint work
with Rognes and Schlichtkrull on \emph{logarithmic topological
  Hochschild homology}~\cite{RSS_LogTHH-I, RSS_LogTHH-II}. For example,
the group completion functor is already used in the definition of
the logarithmic topological Hochschild homology of graded pre-log ring
spectra, and it is a key ingredient for showing that the inclusion of
the Adams summand extends to a map of graded pre-log ring spectra
which is formally {\'e}tale through the eyes of logarithmic $\THH$ and
$\mathrm{TAQ}$. We refer to Rognes' ICM talk~\cite[Section
7]{Rognes-ICM} for a survey of these results.

\subsection{Organization} We begin with a review of diagram spaces in
Section~\ref{sec:diagram-spaces-and-units} and use this material in
Section~\ref{sec:Gamma-spaces-from-monoids} to define the
$\Gamma$-space $\gloneJof{E}$ appearing in the
introduction. Section~\ref{sec:spectra-of-units} contains the proof of
Theorem~\ref{thm:bglonegraded-introduction}. In
Section~\ref{sec:group-compl-model-str} we give a more detailed
formulation of
Theorem~\ref{thm:gp-model-str-and-comparison-introduction} and show
how it can be used to prove
Theorem~\ref{thm:sp-of-graded-units-as-righ-adj-introduction}.
Sections~\ref{sec:GammaJ-spaces} and~\ref{sec:stable-model-str} are
devoted to the construction and identification of group completions
needed to prove
Theorem~\ref{thm:gp-model-str-and-comparison-introduction}.
\subsection{Conventions} 
In this paper, \emph{space} means unpointed simplicial set. Most
proofs would work equally well for compactly generated weak Hausdorff
spaces, although some arguments in connection with localizations of
model categories differ.  When using the symbols $\rightleftarrows$
and $\leftrightarrows$ for (possibly unlabeled) adjoint pairs of
functors, we always denote the left adjoint on the top. We assume some
familiarity with model categories and mostly use Hirschhorn's
book~\cite{Hirschhorn_model} as a reference.
\subsection{Acknowledgments} The author would like to thank Christian
Schlichtkrull and Stefan Schwede for helpful conversations related to
this project, and John Rognes for comments on an earlier version of the manuscript.
Moreover, the author likes to thank an unnamed referee for helpful comments on 
an earlier version of this manuscript.  
\section{Structured diagram spaces and units}\label{sec:diagram-spaces-and-units}
In this section we review some terminology about diagram spaces
from~\cite{Sagave-S_diagram}.  For a commutative symmetric ring
spectrum $E$, we define its units $\GLoneIof{E}$, its graded units
$\GLoneJof{E}$, and the diagram categories in which these objects are
defined.

\subsection{Diagram spaces} We begin with some general definitions
that will be used frequently throughout the paper.
\begin{definition}\label{def:K-space}
  Let $\cK$ be a small category. A \emph{$\cK$-space} is a functor
  $\cK \to \cS$ to the category of (unpointed) simplicial sets. We write
  $\cS^{\cK}$ for the category of $\cK$-spaces.
\end{definition}

A map of $\cK$-spaces $f \colon X \to Y$ is a \emph{level equivalence}
if $f(\bld{k})\colon X(\bld{k}) \to Y(\bld{k})$ is a weak equivalence
for every object $\bld{k}$ of $\cK$. Level equivalences are often too
rigid, and the following weaker notion of equivalence is useful for
many purposes.
\begin{definition}\label{def:K-equivalences}
  A map of $f\colon X \to Y$ of $\cK$-spaces is a $\cK$-equivalence if
  it induces a weak equivalence $f_{h\cK} \colon X_{h\cK} \to
  Y_{h\cK}$ of homotopy colimits. Here
  \[ X_{h\cK}=\hocolim_{\cK}X = \diag\left( [i] \mapsto
    \displaystyle\coprod_{\bld{k_0} \ot \dots \ot \bld{k_i}}
    X(\bld{k_i})\right)\] is the usual Bousfield-Kan homotopy colimit of a
  $\cK$-space $X$.
\end{definition}
Now let $\cK$ be a small symmetric monoidal category with product
$\concat$ and monoidal unit~$\bld{0}$.  Then the category of
$\cK$-spaces has a symmetric monoidal product $\boxtimes$ induced by
the symmetric monoidal structure of $\cK$ and the cartesian product of
spaces: For $\cK$-spaces $X$ and $Y$, we define $X \boxtimes Y$ as the
left Kan extension of the object-wise cartesian product along $\concat
\colon \cK \times \cK \to \cK$.  The unit for $\boxtimes$ is the
$\cK$-space $\bld{1}_{\cK} = \cK(\bld{0},-)$.

\begin{definition}
  A \emph{commutative $\cK$-space monoid} is a commutative monoid in
  $(\cS^{\cK},\boxtimes, \eins_{\cK})$, and $\cC\cS^{\cK}$ denotes the
  category of commutative $\cK$-space monoids.
\end{definition}

Unraveling the definition of a left Kan extension, a commutative
$\cK$-space monoid is a $\cK$-space $A$ together with multiplication
maps $A(\bld{k})\times A(\bld{l}) \to A(\bld{k}\concat\bld{l})$ and a
unit map $* \to A(\bld{0})$ which are associative, unital, and
commutative.

For a commutative $\cK$-space monoid $A$, the homotopy colimit $A_{h\cK}$
inherits the structure of a simplicial monoid with product
\[ A_{h\cK} \times A_{h\cK} \iso
\hocolim\displaylimits_{(\bld{k},\bld{l})\in \cK \times \cK}
A(\bld{k})\times A(\bld{l}) \to
\hocolim\displaylimits_{(\bld{k},\bld{l})\in \cK \times \cK}
A(\bld{k}\concat\bld{l}) \to A_{h\cK}. \] Here the first map is induced by
the multiplication of $A$ and the second map induced by the monoidal
structure of $\cK$. The monoid $A_{h\cK}$ is usually not commutative
since the symmetry isomorphism of $\cK$ may differ from the
identity. However, commutativity of $A$ implies that the monoid of path
components $\pi_0(A_{h\cK})$ of $A_{h\cK}$ is commutative.

\begin{definition}\label{def:units-k-space-monoid}
  Let $A$ be a commutative $\cK$-space monoid. \begin{enumerate}[(i)]
  \item $A$ is \emph{grouplike} if the commutative monoid
    $\pi_0(A_{h\cK})$ is a group. 
  \item The units $A^{\times}$ of $A$ is the grouplike sub commutative
    $\cK$-space monoid of $A$ with $A^{\times}(\bld{k})$ consisting of
    those components of $A(\bld{k})$ that map to a unit in the
    commutative monoid $\pi_0(A_{h\cK})$.
  \end{enumerate}
\end{definition}
One can check that $A^{\times}$ is well defined and that $A^{\times} \to A$ realizes
the inclusion $(\pi_0(A_{h\cK}))^{\times} \to \pi_0(A_{h\cK})$. 

\subsection{\texorpdfstring{$\cI$}{I}-spaces and units} The category
of finite sets and injections $\cI$ is one indexing category for
diagram spaces which is of importance for us.
\begin{definition}
  Let $\cI$ be the category with objects the sets
  $\bld{m}=\{1,\dots,m\}$ for $m \geq 0$ and morphisms the injective
  maps. The ordered concatenation $\concat$ of ordered sets makes
  $\cI$ a permutative category, i.e., a symmetric monoidal category
  with strict unit and associativity. The monoidal unit is the empty set
  $\bld{0}$ and the symmetry isomorphism the shuffle $\chi_{m,n}\colon
  \bld{m}\concat\bld{n} \to \bld{n}\concat \bld{m}$ moving the first
  $m$ elements past the last $n$ elements.
\end{definition}

We obtain a category $\cC\cS^{\cI}$ of \emph{commutative $\cI$-space
  monoids}. This category is useful because it admits a \emph{positive
  $\cI$-model structure} such that $\cC\cS^{\cI}$ is Quillen
equivalent to the category of $E_{\infty}$ spaces over an $E_{\infty}$
operad~\cite[Theorem 1.2]{Sagave-S_diagram}. The weak equivalences in
the positive $\cI$-model structure are the $\cI$-equivalences of
Definition~\ref{def:K-equivalences}. So all homotopy types of
$E_{\infty}$ spaces are represented by commutative $\cI$-space
monoids. In many examples one can write down explicit models for such
commutative $\cI$-space monoids. We now review
from~\cite{Schlichtkrull_units,Sagave-S_diagram} how this is done
for the units of a commutative symmetric ring spectrum.

A commutative symmetric ring spectrum $E$ may be described by a
sequence of pointed spaces $E_n$ for $n \geq 0$ with base point
preserving left $\Sigma_n$-actions on $E_n$, multiplication maps
$E_{m} \sm E_{n} \to E_{m+n}$, and unit maps $S^n \to E_n$ satisfying
appropriate associativity, commutativity, and unitality
conditions~\cite[I.1]{Schwede_SymSp}.  In this
description, the iterated structure map $\sigma\colon E_m \sm S^n \to
E_{m+n}$ of the underlying symmetric spectrum of $E$ is the composite
$E_m \sm S^n \to E_m \sm E_n \to E_{m+n}$ of the unit and the
multiplication. We write $\cC\Spsym$ for the category of commutative
symmetric ring spectra.

The commutative $\cI$-space monoid $\OmegaIof{E}$ modeling the
underlying multiplicative $E_{\infty}$ space of $E$ is defined as
follows. Its value at an object $\bld{m}$ in $\cI$ is the space
$\Omega^{m}(E_{m})$. To define the action of a morphism $\alpha\colon
\bld{m} \to \bld{n}$ in $\cI$, choose an extension $\pi \colon \bld{n}
\to \bld{n}$ of $\alpha$ to a bijection and define $\alpha_*(f)$ for a
zero simplex $f\colon S^m \to E_m$ in $\Omega^m(E_m)$ to be the composite
\[ S^{n} \xrightarrow{\pi^{-1}_*} S^{n} \iso S^m \sm S^{n-m}
\xrightarrow{f \sm \id} E_m \sm S^{n-m} \xrightarrow{\sigma} E_{n}
\xrightarrow{\pi_{*}} E_{n}.\] This does not depend on the choice of
$\pi$ and extends to a morphism of mapping spaces $\Omega^{m}(E_m) \to
\Omega^{n}(E_{n})$ that makes $\OmegaIof{E}$ an $\cI$-space. The
multiplication of $E$ defines maps $\Omega^{m}(E_m) \times
\Omega^{n}(E_n) \to \Omega^{m+n}(E_{m+n})$ by sending $f\colon S^m \to
E_m$ and $g\colon S^n \to E_n$ to
\[ S^{m+n}\iso S^{m} \sm S^{n} \xrightarrow{f \sm g} E_m \sm E_n \to
E_{m+n},\] and $\OmegaIof{E}$ becomes a commutative $\cI$-space
monoid with this structure.

One can only expect $\OmegaIof{E}$ to be homotopically well behaved
for sufficiently fibrant~$E$. (Already the mapping spaces involved are
not homotopy invariant for general simplicial sets). We will therefore
only apply $\OmegaI$ to fibrant objects in the positive model
structure on $\cC\Spsym$~\cite{MMSS}. By definition, $E$ is positive
fibrant if for $m\geq 1$, the $E_m$ are fibrant as a simplicial sets and
the adjoint structure maps $E_m \to \Omega(E_{m+1})$ are weak
equivalences. In fact, $\OmegaI$ is a right Quillen functor with
respect to the positive model structure on $\cC\Spsym$ and the
\emph{positive $\cI$-model structure} on
$\cC\cS^{\cI}$~\cite[Proposition 3.19(ii)]{Sagave-S_diagram}.
\begin{definition}\label{def:I-units-ring-spectrum}
  Let $E$ be a positive fibrant commutative symmetric ring
  spectrum. The \emph{units} $\GLoneIof{E}$ of $E$ is the commutative
  $\cI$-space monoid $(\OmegaIof{E})^{\times}$. 
\end{definition}
For such $E$, the fact that the classifying space $B\cI$ is
contractible can be used to show that the commutative monoid
$\pi_0((\OmegaIof{E})_{h\cI})$ is isomorphic to the underlying
multiplicative monoid of the $0$th stable homotopy group $\pi_0(E)$ of
$E$. The sub commutative $\cI$-space monoid $\GLoneIof{E}$ is
characterized by $\pi_0(\GLoneIof{E})_{h\cI} \iso
(\pi_0(E))^{\times}$. It provides the desired strictly commutative
model for the units of $E$.

\subsection{\texorpdfstring{$\cJ$}{J}-spaces and graded
  units}\label{subsec:graded-units} Our second example of an indexing category 
for diagram spaces will lead to the graded units outlined in the
introduction.

\begin{definition}[{\cite[Definition 4.2]{Sagave-S_diagram}}]
Let $\cJ$ be the category whose objects  are
pairs $(\bld{m_1},\bld{m_2})$ of objects in $\cI$. There are no morphisms
$(\bld{m_1},\bld{m_2})\to (\bld{n_1},\bld{n_2})$ unless $m_2-m_1 = n_2
-n_1$, and in this case a morphism $(\bld{m_1},\bld{m_2})\to
(\bld{n_1},\bld{n_2})$ is a triple $(\beta_1, \beta_2, \sigma)$ with
the $\beta_i\colon \bld{m_i}\to\bld{n_i}$ morphisms in $\cI$ and
$\sigma \colon \bld{n_1}\setminus\beta_1(\bld{m_1}) \to
\bld{n_2}\setminus\beta_2(\bld{m_2})$ a bijection identifying the
complements of $\beta_1$ and $\beta_2$.
\end{definition}
This $\cJ$ is indeed a category with the composite
\[ (\bld{l_1},\bld{l_2}) \xrightarrow{(\alpha_1,\alpha_2,\rho)}
(\bld{m_1},\bld{m_2})\xrightarrow{(\beta_1, \beta_2, \sigma)}
(\bld{n_1},\bld{n_2}) \] defined to be the composite $\beta_i
\alpha_i$ in the first two components and the bijection $\sigma \cup
\beta_2\rho\beta_1^{-1}$ in the last component. Concatenation in both entries makes $\cJ$ a permutative
category.

As before let $E$ be a commutative symmetric ring spectrum. The
category $\cJ$ is made up so that there is a commutative $\cJ$-space
monoid $\OmegaJof{E}$ whose value at an object
$(\bld{m_1},\bld{m_2})$ is the space $\Omega^{m_2}(E_{m_1})$. To
describe the action of a morphism $(\beta_1, \beta_2, \sigma)\colon
(\bld{m_1},\bld{m_2})\to (\bld{n_1},\bld{n_2})$ in $\cJ$, we choose
two bijections $\pi_i\colon \bld{n_i} \to \bld{n_i}$ extending the
$\beta_i$ such that $\sigma(\pi_1(m_1 + j)) = \pi_2(m_2 + j)$ for $1
\leq j \leq n_1 - m_1 = n_2 - m_2$. On a zero simplex $f \colon
S^{m_2} \to E_{m_1}$ of $\Omega^{m_2}(E_{m_1})$, the induced map
$(\beta_1, \beta_2, \sigma)_*(f)$ is  defined as the composite
\[ S^{n_2} \xrightarrow{(\pi_2^{-1})_*} S^{n_2} \iso S^{m_2} \sm S^{n_2-m_2}
\xrightarrow{f \sm \id} E_{m_1} \sm S^{n_1-m_1} \xrightarrow{\sigma} E_{n_1}
\xrightarrow{(\pi_1)_{*}} E_{n_1}.\] As in the case of $\cI$-spaces, this
does not depend on the choice of the $\pi_i$ and extends to a morphism
of mapping spaces $ \Omega^{m_2}(E_{m_1}) \to \Omega^{n_2}(E_{n_1})$ that
makes $\OmegaJof{E}$ a $\cJ$-space. The monoid structure on $\OmegaJof{E}$
is defined similarly as for $\OmegaIof{E}$. 

Again $\OmegaJ$ is not homotopy invariant on the whole category
$\cC\Spsym$. Since it is a right Quillen functor with respect to the
positive stable model structure on $\cC\Spsym$ and the \emph{positive
  $\cJ$-model structure} on $\cC\cS^{\cJ}$ to be reviewed in
Section~\ref{sec:group-compl-model-str}, we only apply $\OmegaJ$ to
positive fibrant commutative symmetric ring spectra.

\begin{definition}\label{def:J-units-ring-spectrum}
  Let $E$ be a commutative symmetric ring spectrum. The \emph{graded
    units} $\GLoneJof{E}$ of $E$ is the commutative $\cJ$-space monoid
  $(\OmegaJof{E})^{\times}$.
\end{definition}
The graded units $\GLoneJof{E}$ now captures the units in all degrees
of the graded commutative ring $\pi_*(E)$: In~\cite[Section
4.14]{Sagave-S_diagram} we prove that one can associate a \emph{graded
  signed monoid} $\pi_{0,*}(A)$ with a commutative $\cJ$-space monoid
$A$.  Here a \emph{graded signed monoid} is a $\mZ$-graded monoid
$M_*$ together with $\{\pm 1\}$-actions on each $M_s$ such that $ab =
(-1)^{st} ba$ holds for $a \in M_s$ and $b \in M_t$.   For
$\pi_{0,*}(\OmegaJof{E})$, the $\{\pm 1\}$-action is the additive sign
action of the graded ring $\pi_*(E)$. The inclusion
$\GLoneJof{E} \to \OmegaJof{E}$ realizes the inclusion
$\pi_*(E)^{\times} \to \pi_*(E)$ under $\pi_{0,*}(-)$.

There is an important difference to the situation of $\cI$-spaces. The
graded signed monoid $\pi_0(A)$ of a commutative $\cJ$-space monoid
$A$ is not simply given by $\pi_0(A_{h\cJ})$. As we will see in
Proposition~\ref{prop:low_dim_hty_of_bglonegraded}, the reason behind
this is that the classifying space $B\cJ$ of $\cJ$ is not
contractible, but rather homotopy equivalent to $QS^0$. However, these
two monoids are related: The $\pi_{0}(A_{h\cJ})$ is isomorphic to the
quotient of the underlying ungraded monoid of $\pi_{0}(A)$ by the
$\{\pm 1\}$-action~\cite[Corollary 4.17]{Sagave-S_diagram}.

A commutative $\cJ$-space monoid $A$ has an underlying commutative
$\cI$-space monoid: There is a strong symmetric monoidal diagonal
functor $\Delta\colon \cI \to \cJ$ with
$\Delta(\bld{m})=(\bld{m},\bld{m})$ and $\Delta(\alpha\colon \bld{m}
\to\bld{n})=(\alpha,\alpha, \id_{\bld{n}\setminus \alpha(\bld{m})})$,
and it is easy to see that $\Delta^{*}(A) = A\circ \Delta$ is a
commutative $\cI$-space monoid. This recovers $\GLoneIof{E}$ from
$\GLoneJof{E}$:
\begin{lemma}\label{lem:units-correspond-under-diagonal}
There is a natural isomorphism $\Delta\!^{*}\!(\GLoneJof{E}) \iso \GLoneIof{E}$.
\end{lemma}
\begin{proof}
  The isomorphism $\Delta\!^{*}(\OmegaJof{E}) \iso \OmegaIof{E}$ is
  evident, and a $0$-simplex in $\Omega^{m}(E_m)$ maps to an
  invertible path component of $(\OmegaJof{E})_{h\cJ}$ if and only if
  it maps to an invertible path component of $(\OmegaIof{E})_{h\cI}$.
\end{proof}

\section{\texorpdfstring{$\Gamma$}{Gamma}-spaces associated 
with structured diagram spaces}\label{sec:Gamma-spaces-from-monoids}
The aim of this section is to build $\Gamma$-spaces whose associated
spectra provide deloopings of the simplicial monoids
$(\GLoneIof{E})_{h\cI}$ and $(\GLoneJof{E})_{h\cJ}$. The construction
will involve deloopings of the respective indexing categories that we
construct first. Since it covers the relevant examples and simplifies
the bookkeeping, we restrict ourselves to permutative indexing
categories.
\subsection{\texorpdfstring{$\Gamma$}{Gamma}-spaces}\label{subsec:GammaS}
We begin with recalling some basic terminology about $\Gamma$-spaces
from~\cite{Segal_categories}
and~\cite{Bousfield-F_Gamma-bisimplicial}. Let $\Gamma^{\op}$ be the
category of finite based sets. We write $k^+$ for the object of
$\Gamma^{\op}$ given by the set $\{0,\dots,k\}$ with basepoint $0$.  A
\emph{$\Gamma$-space} is a covariant functor $X\colon \Gamma^{\op} \to
\cS$ such that $X(0^+) = *$. The $\Gamma$-space $X$ is called
\emph{special} if for each pair of finite based sets $(S,T)$ the
natural map $X(S \wdg T) \to X(S) \times X(T)$ induced by the
projections $S\wdg T \to S$ and $S\wdg T \to T$ is a weak
equivalence. For a special $\Gamma$-space $X$ the set of path
components $\pi_0(X(1^+))$ has the structure of a commutative monoid,
and $X$ is called \emph{very special} if this monoid is a group. A
$\Gamma$-space $X$ has an associated (sequential) spectrum $X(\mS)$
with $X(\mS)_0 = X(1^+)$. If $X$ is very special, the level fibrant
replacement of $X(\mS)$ is an $\Omega$-spectrum and therefore exhibits
$X(\mS)_0 = X(1^+)$ as an infinite loop space.

\subsection{Delooping permutative categories}\label{subsec:delooping-perm} 
We now explain how to deloop the classifying space $B\cK$ of a small
permutative category.  The idea is to build a special $\Gamma$-space
from $\cK$ whose associated spectrum is a delooping of the simplicial
monoid $B\cK$ if $\pi_0(B\cK)$ is a group. The naive choice would be
to try to define a $\Gamma$-space by $S \mapsto B(\cK^{\times
  \ovl{S}})$ where $\cK^{\times \ovl{S}}$ denotes a product of copies
of $\cK$ indexed by the set $\ovl{S} = S \setminus 0^+$.  However,
this cannot be made functorial in $\Gamma^{\op}$: The objects of
$\Gamma^{\op}$ are unordered sets while a non-trivial symmetry
isomorphism of $\cK$ makes it necessary to specify an ordering of the
factors before using the monoidal product to define the structure maps
of the $\Gamma$-space. The following construction by Shimada and
Shimakawa (see~\cite{Shimada-Shimakawa_delooping} or~\cite[II, \S
3.1]{Dundas_GMc_local}) avoids this problem.

\begin{definition}\label{def:Shi-Shi-construction} Let $(\cK,\concat,\bld{0})$ be a permutative
  category. For an object $S$ in $\Gamma^{op}$, let $\HcK(S)$ be the
  category whose objects are maps $\bld{s} \colon \cP(\ovl{S}) \to
  \obj(\cK)$ from the power set of $\ovl{S} = S \setminus 0^+$ to the
  set of objects of $\cK$ together with isomorphisms $\sigma_{U,V}
  \colon \bld{s}_{U}\concat \bld{s}_{V} \to \bld{s}_{U\cup V}$ for
  every pair of disjoint subsets $U,V \subseteq \ovl{S}$ such that the
  following conditions hold:
\begin{enumerate}[(i)]
\item $\bld{s}_{\emptyset} = \bld{0}$
\item $\sigma_{\emptyset, V}\colon \bld{0}\concat \bld{s}_{V} \to
  \bld{s}_{V}$ and $\sigma_{U,\emptyset} \colon \bld{s}_{U}\concat
  \bld{0} \to \bld{s}_{U}$ are the identity
\item $\sigma_{U\cup V,W} \circ (\sigma_{U,V} \concat \bld{s}_W) =
  \sigma_{U,V\cup W} \circ (\bld{s}_U \concat \sigma_{V,W})$ as
  maps $\bld{s}_U \concat \bld{s}_V \concat \bld{s}_W \to
  \bld{s}_{U\cup V\cup W}$
\item $\sigma_{V,U} \circ \chi_{\bld{s}_U,\bld{s}_V} = \sigma_{U,V}$
  as maps $\bld{s}_U \concat \bld{s}_V \to \bld{s}_{U\cup V}$
\end{enumerate}
Morphisms $f \colon (\bld{s},\sigma) \to (\bld{t},\tau)$ are families
of morphisms $f_U\colon \bld{s}_U\to\bld{t}_U$ such that
$f_{\emptyset} = \id_{\bld{0}}$ and $\tau_{U,V} \circ (f_{U}\concat
f_{V}) = f_{U\cup V} \circ \sigma_{U,V}$ hold.
\end{definition}
It is easy to see that $\HcK(S)$ is again a permutative category. Its 
monoidal unit~$\bld{0}_S$ is the map $\cP(\ovl{S}) \to \obj(\cK)$ with 
constant value $\bld{0}$. 

For a map $\alpha \colon S \to T$ in $\Gamma^{\op}$ and an object
$(\bld{s},\sigma)$ in $\HcK(S)$, we set
\[ \alpha_*(\bld{s}) = \left(\cP(\ovl{T}) \xrightarrow{\alpha^{-1}}
  \cP(\ovl{S}) \xrightarrow{\bld{s}} \obj(\cK)\right) \] and
$\alpha_*(\sigma)_{U,V} = \sigma_{\alpha^{-1}(U),\alpha^{-1}(V)}$ in
order to define an object $\alpha_*(\bld{s},\sigma)$ of $\HcK(T)$.
\begin{lemma}\label{lem:Shi-Shi-functorial}
  This defines a covariant functor $\HcK$  from $\Gamma^{\op}$
  to the category of categories. Evaluation of an object at the one
  element subsets of $S$ induces an equivalence of categories
  $\HcK(S) \to \cK^{\times\ovl{S}}$.
\end{lemma}
\begin{proof}
  The functoriality is immediate. The choice of an ordering of
  $\ovl{S}$ induces an inverse to the functor $\HcK(S) \to
  \cK^{\times\ovl{S}}$.
\end{proof}

\begin{definition}
  We write $\bof{\cK}$ for the $\Gamma$-space obtained by composing
  $\HcK$ with the classifying space functor.
\end{definition}
The $\Gamma$-space $\bof{\cK}$ is also known as the \emph{algebraic $K$-theory} of
$\cK$~\cite[II, \S 3.1]{Dundas_GMc_local}.
\begin{corollary}
  The $\Gamma$-space $\bof{\cK}$ is special. If the monoid of path
  components $\pi_0(B\cK)$ of $\cK$ is a group, then $\bof{\cK}$ is
  very special.\qed
\end{corollary}

\begin{example}
  The very special $\Gamma$-space $\bof{\cI}$ associated with the
  category $\cI$ is not very interesting: It is levelwise contractible
  because $\cI$ has an initial object.
\end{example}
\begin{example}\label{ex:spectrum-of-J}
  It is shown in~\cite[Proposition 4.4]{Sagave-S_diagram} that the
  category $\cJ$ is equivalent to Quillen's localization construction
  $\Sigma^{-1}\Sigma$ on the category of finite sets and bijections
  $\Sigma$. Hence its classifying space $B\cJ$ is equivalent to $QS^0$
  by the Barratt-Priddy-Quillen theorem. Segal's version of the
  Barratt-Priddy-Quillen theorem~\cite{Segal_categories} implies that
  the special $\Gamma$-space $\bof{\Sigma}$ is stably equivalent
  to the $\Gamma$-space $\Gamma^{\op}(1^+,-)$ representing the sphere
  spectrum. Since the functor $\Sigma \to \Sigma^{-1}\Sigma = \cJ$
  induces a group completion $B\Sigma \to B\cJ$, the map $\bof{\Sigma}
  \to \bof{\cJ}$ is a stable equivalence of $\Gamma$-spaces. So the
  very special $\Gamma$-space $\bof{\cJ}$ is represents the sphere
  spectrum. The associated spectrum $\bof{\cJ}(\mS)$ is thus (up to
  level fibrant replacement) an $\Omega$-spectrum representing the
  homotopy type of the sphere spectrum. As opposed to other
  $\Omega$-spectra modeling the sphere spectrum, its definition does
  not involve taking a loop space or forming a telescope.
\end{example}

\subsection{\texorpdfstring{$\Gamma$}{Gamma}-spaces from commutative \texorpdfstring{$\cK$}{K}-space monoids}\label{subsec:Gamma-from-CSK}
In~\cite{Schlichtkrull_units}, Schlichtkrull constructs
$\Gamma$-spaces from commutative $\cI$-space monoids. We give a
description of his approach that applies to both commutative $\cI$-
and $\cJ$-space monoids.

Let $\cK$ be a small permutative category, let $A$ be a commutative $\cK$-space
monoid, and let $\HcK$ be the $\Gamma$-category associated with $\cK$. For
every object $S$ in $\Gamma^{\op}$, we define a $\HcK(S)$-space $A_S$ by
\[ (\bld{s},\sigma) \mapsto \textstyle\prod_{i\in \ovl{S}} A(\bld{s}_i). \]
Here we write $\bld{s}_i$ for $\bld{s}(\{i\})$. 

We saw that a map $\alpha\colon S \to T$ in $\Gamma^{\op}$ induces a
functor $\alpha_* \colon \HcK(S) \to \HcK(T)$. There
is an induced functor $\alpha^*\colon \cS^{\HcK(T)}\to\cS^{\HcK(S)}$
sending $A_T$ to $\alpha^*(A_T)=A_T\circ \alpha_*$.  The value of
$\alpha^*(A_T)$ at an object $(\bld{s},\sigma)$ of $\HcK(S)$ is
\[ \textstyle\prod_{j\in \ovl{T}} A(\bld{s}(\alpha^{-1}(j))). \] 

The structure maps of the commutative $\cK$-space monoid $A$
induce a map
\begin{equation}\label{eq:map-of-HKS-diagrams}
  \widetilde{\alpha} \colon A_S \to \alpha^*(A_T)
\end{equation} of $\HcK(S)$-spaces: For every
element $j \in \ovl{T}$, we choose an ordering of the set
$V=\alpha^{-1}(j)$ and define the map to be the product over all $j
\in \ovl{T}$ of the maps
\begin{equation}\label{eq:map-of-HKS-diagrams-components} 
  \textstyle\prod_{i\in V} A(\bld{s}_i) \to
  A(\concat_{i\in V}\bld{s}_i) \to A(\bld{s}_V).
\end{equation}
Here the iterated monoidal product in $\cK$ in the middle term is
formed using the chosen ordering of $V$, the first map comes from the
product of $A$, and the last map is induced by the isomorphisms
$\sigma$ that belong to the object of $\HcK(S)$ in question. The
commutativity of $A$ ensures that the
map~\eqref{eq:map-of-HKS-diagrams-components} does not depend on the
choice of the ordering. If $\alpha^{-1}(j)$ is empty, the
map~\eqref{eq:map-of-HKS-diagrams-components} is the unit $* \to
A(\bld{0})$ of $A$.
\begin{lemma}\label{lem:compatibilibty-of-AS}
  For maps $\alpha\colon S \to T$ and $\beta\colon T \to U$ in
  $\Gamma^{op}$, the composite $ (\alpha^*(\widetilde{\beta}))
  \widetilde{\alpha}$ and $\widetilde{\beta\alpha}$ coincide as maps
  $A_S \to \alpha^* (\beta^* A_U)= (\beta \alpha)^* A_U$.  \qed
\end{lemma}

\begin{definition}\label{def:Gamma-space-from-K-space-monoid}
  The $\Gamma$-space $\gamma(A)$ associated with the commutative
  $\cK$-space monoid $A$ is defined by $\gamma(A)(S) =
  \hocolim_{\HcK(S)}A_S$ on the objects of $\Gamma^{\op}$. The
  structure map associated with $\alpha\colon S \to T$ in
  $\Gamma^{\op}$ is the composite
  \[\hocolim_{\HcK(S)}A_S \to \hocolim_{\HcK(S)}(\alpha^*(A_T)) \to
  \hocolim_{\HcK(T)}A_T \] induced by the map $\widetilde{\alpha}$
  and the map of homotopy colimits induced by $\alpha_*$.
\end{definition}
The previous lemma ensures that $\gamma(A)$ has indeed the structure maps of
a $\Gamma$-space. By definition, there is an isomorphism
$\gamma(A)(1^+) \iso A_{h\cK}$.

It is clear that $\gamma(A)$ is natural in $A$.  Let $*$ be the
terminal commutative $\cK$-space monoid. Since $\hocolim_{\cK}(*) \iso
B\cK$, it follows that $\gamma(*) \iso \bof{\cK}$. We obtain
\begin{corollary}\label{cor:hocolim-as-space-over-BK}
  There is a canonical natural map $\gamma(A) \to \bof{\cK}$ of
  $\Gamma$-spaces.\qed
\end{corollary}
We have the following analogue of~\cite[Proposition
5.3]{Schlichtkrull_units}:
\begin{proposition}\label{prop:gamma-A-special}
  The $\Gamma$-space $\gamma(A)$ is special, and it is very special if $A$
  is grouplike. For general $A$ there is a weak equivalence $B(A_{h\cK})\to \gamma(A)(S^1)= \gamma(A)(\mS)_1$. 
\end{proposition}
\begin{proof}
  The proof of~\cite[Proposition 5.3]{Schlichtkrull_units} applies almost
verbatim: The equivalence 
\[ 
\gamma(A)(S) = \hocolim_{\HcK(S)}A_S \to \hocolim_{\cK^{\times \ovl{S}}}A^{\times
  \ovl{S}} \iso \textstyle\prod_{\ovl{S}} A_{h\cK}
\]
show that $\gamma(A)$ is special. It is easy to see that the condition
of $A$ being grouplike and $\gamma(A)$ being very special refer to the
same monoid structure on $\pi_0(A_{h\cK})$. The equivalence of the
evaluation at the sphere $S^1$ with the bar construction follows as
in~\cite[Proposition 5.3]{Schlichtkrull_units} from the choice of an
ordering of the simplices in $S^1$.
\end{proof}

The motivating examples for this construction come from the diagram
space models for units of ring spectra introduced in the previous
section.
\begin{definition}\label{def:units-graded-units}
  Let $E$ be a commutative symmetric ring spectrum. The \emph{units}
  of $E$ is the $\Gamma$-space $\gloneIof{E} = \gamma(\GLoneIof{E})$
  associated with the commutative $\cI$-space monoid
  $\GLoneIof{E}$. The \emph{graded units} is the $\Gamma$-space
  $\gloneJof{E}=\gamma(\GLoneJof{E})$ associated with the commutative
  $\cJ$-space monoid $\GLoneJof{E}$. We view these constructions as
  functors
\[ \gloneI \colon \cC\Spsym \to \GammaS \quad \text{ and } \quad   \gloneJ \colon \cC\Spsym \to \GammaS / \bof{\cJ}. \]
\end{definition}
By construction, both $\gloneIof{E}$ and $\gloneJof{E}$ are very
special. While the augmentation $\gloneJof{E} \to \bof{\cJ}$ will be
central in what follows, we will usually ignore the augmentation of
$\gloneIof{E}$ to the levelwise contractible $\Gamma$-space
$\bof{\cI}$. Both functors will only be homotopically well behaved for
sufficiently fibrant (e.g. positive fibrant) $E$ since this was
already the case with $\GLoneIof{E}$ and $\GLoneJof{E}$.

\begin{remark}
  There is a classical notion of a spectrum of units associated with
  an $E_{\infty}$ spectrum due to May and
  coauthors~\cite{May_ring_spaces}. The spectrum associated with
  $\gloneIof{E}$ is a possible construction of this object in the
  context of commutative symmetric ring spectra. It was introduced by
  Schlichtkrull~\cite{Schlichtkrull_units}. Lind~\cite{Lind-diagram}
  shows that $(\gloneIof{E})(\mS)$ is indeed equivalent to the spectra
  of units associated with other kinds of spectra.
\end{remark}

We note that the $\Gamma$-spaces considered here are natural in the indexing category:
\begin{lemma}\label{lem:Gamma-space-map-from-space-of-cat}
  Let $F\colon \cL \to \cK$ be a strong symmetric monoidal functor of
  small permutative categories and let $A$ be a commutative
  $\cK$-space monoid. Then $F^*(A)$ is a commutative $\cL$-space
  monoid, and there is a natural map $\gamma(F^*A) \to \gamma(A)$ of
  $\Gamma$-spaces.  \qed
\end{lemma}

\section{Spectra of units}\label{sec:spectra-of-units}
Throughout this section we let $E$ be a positive fibrant commutative
symmetric ring spectrum. We will compare the various spectra of units
associated with $E$ and prove Theorem~\ref{thm:bglonegraded-introduction}
from the introduction. 

We begin by applying Lemma~\ref{lem:Gamma-space-map-from-space-of-cat}
and Lemma~\ref{lem:units-correspond-under-diagonal} to the units and
the graded units of Definition~\ref{def:units-graded-units} in order
to obtain a sequence of $\Gamma$-spaces
\begin{equation}\label{eq:seq-of-Gamma-sp}
  \gloneIof{E} \to \gloneJof{E} \to \bof{\cJ}.
\end{equation}

\begin{proposition}\label{prop:hty-fiber-seq-Gamma-spaces}
  The sequence~\eqref{eq:seq-of-Gamma-sp} is a homotopy fiber sequence
  of $\Gamma$-spaces.
\end{proposition}
\begin{remark}\label{rem:hty-fiber-squares}
  By definition, we call a sequence of $\Gamma$-spaces $ W \to X
  \xrightarrow{g} Y$ a \emph{homotopy fiber sequence} if $W$ maps by a
  stable equivalence into the base change of $\ovl{g}$ along the map $*
  \to Y$ where $\ovl{g}\colon \ovl{X} \to Y$ is a replacement of $g$
  by a stable fibration. All model category notions in this statement
  refer to the stable $Q$-model structure on $\Gamma$-spaces that we
  review in Section~\ref{sec:group-compl-model-str}.

  Some care is needed here because this model structure on
  $\Gamma$-spaces fails to be right proper~\cite[\S
  5.7]{Bousfield-F_Gamma-bisimplicial} and the above definition only
  provides a well defined notion of homotopy fiber squares in a right
  proper model category~\cite[II.8]{Goerss-J_simplicial}. There is no
  problem in the case at hand because all $\Gamma$-spaces
  in~\eqref{eq:seq-of-Gamma-sp} are (up to a level fibrant
  replacement) stably fibrant. This is one reason for why we do not try to
  replace the $\Gamma$-space $\bof{\cJ}$ by the stably equivalent
  $\Gamma$-space $\Gamma^{\op}(1^+,-)$ whose associated spectrum is
  the sphere spectrum $(S^0, S^1, \dots)$.
\end{remark}
\begin{proof}[Proof of Proposition~\ref{prop:hty-fiber-seq-Gamma-spaces}]
  A general fact about left Bousfield localizations of model
  categories~\cite[Proposition 3.4.7]{Hirschhorn_model} shows that a
  levelwise fibration between very special $\Gamma$-spaces is a
  fibration in the stable model structure on $\Gamma$-spaces. (This
  uses right properness of the level model structure, but not of the
  stable model structure.) So it is enough to show that the
  sequence~\eqref{eq:seq-of-Gamma-sp} gives a homotopy fiber sequence
  of spaces when evaluated at any object $S$ of $\Gamma^{\op}$. Since
  all $\Gamma$-space involved are special, it is enough to check this
  for $S = 1^+$.

  Because $\GLoneJof{E}$ is positive fibrant,~\cite[Lemma
  4.12]{Sagave-S_diagram} (which is in turn based on~\cite[Lemma
  IV.5.7]{Goerss-J_simplicial} leading to Quillen's theorem B) shows
  that the commutative square
  \begin{equation}\label{eq:BJ-hJ-hty-pullback}
    \xymatrix@-1pc{ \left(\GLoneJof{E}\right)(\bld{m},\bld{m}) \ar[r] \ar[d] 
      & \left(\GLoneJof{E}\right)_{h\cJ} \ar[d] \\
      \{(\bld{m},\bld{m})\} \ar[r] & B\cJ }\end{equation} 
  is homotopy cartesian if $m\geq 1$.  Since $B\cI$ is contractible,
  the same argument with $\cJ$ replaced by $\cI$ shows that
  $\left(\GLoneIof{E}\right)(\bld{m}) \to
  \left(\GLoneIof{E}\right)_{h\cI}$ is a weak equivalence. 
  By Lemma~\ref{lem:units-correspond-under-diagonal},
  $\left(\GLoneJof{E}\right)(\bld{m},\bld{m}) \iso
  \left(\GLoneIof{E}\right)(\bld{m})$, and the composite of the resulting
  map $\left(\GLoneJof{E}\right)(\bld{m},\bld{m}) \to (\GLoneIof{E})_{h\cI}$
  with the map $(\GLoneIof{E})_{h\cI} \to (\GLoneJof{E})_{h\cJ}$ is the top horizontal
  map in~\eqref{eq:BJ-hJ-hty-pullback}.
  Hence $(\GLoneIof{E})_{h\cI}$ maps
  by a weak equivalence into the homotopy fiber of $(\GLoneJof{E})_{h\cJ} \to B\cJ$
  over $(\bld{m},\bld{m})$. Since there is a
  morphism $(\bld{0},\bld{0}) \to (\bld{m},\bld{m})$ in $\cJ$, the same
  is true for $(\bld{0},\bld{0})$ instead of $(\bld{m},\bld{m})$. This verifies the claim because the object $(\bld{0},\bld{0})$ is the basepoint of $B\cJ$.
\end{proof}
\begin{definition}
  We define $\bglonegradedof{E}$ to be the homotopy cofiber of the
  map of spectra associated with the map of $\Gamma$-spaces
  $\gloneJof{E} \to \bof{\cJ}$.
\end{definition}
This $\bglonegradedof{E}$ is an interesting object for the following
reason: Since the map $\pi_0(\gloneJof{E}) \to \pi_0(\bof{\cJ}) \iso
\pi_0(\mS) \iso \mZ$ is in general not surjective, the homotopy fiber
sequence~\eqref{eq:seq-of-Gamma-sp} does not necessarily induce a
homotopy fiber sequence of spectra. If it would give a homotopy fiber
sequence of spectra, the spectrum $\bglonegradedof{E}$ would simply be
the suspension of the spectrum $(\gloneIof{E})(\mS)$ associated with
$\gloneIof{E}$. This suspension is usually denoted by~$\bgloneof{E}$.

The failure of preserving homotopy fiber sequences means that
$\bglonegradedof{E}$ and $\bgloneof{E}$ differ in general. The point
of Theorem~\ref{thm:bglonegraded-introduction} is that the ``graded''
version $\bglonegradedof{E}$ defined here extends the usual
$\bgloneof{E}$ in an interesting way:

\begin{proof}[Proof of Theorem~\ref{thm:bglonegraded-introduction}]
  Since the spectra associated with $\gloneJof{E}$ and $\bof{\cJ}$ are
  connective, $\bglonegradedof{E}$ is connective. The right adjoint of
  $X \mapsto X(\mS)$ is a right Quillen functor and may be viewed as a
  model for the connective cover~\cite[\S
  5]{Bousfield-F_Gamma-bisimplicial}. Rewriting the homotopy cofiber
  sequence defining $\bglonegradedof{E}$ as a homotopy fiber sequence of
  spectra
  \[\Omega(\bglonegradedof{E}) \to (\gloneJof{E})(\mS) \to
  (\bof{\cJ})(\mS),\] we can apply the right adjoint and compare with
  the homotopy fiber sequence~\eqref{eq:seq-of-Gamma-sp} to see that
  the connective cover of $\Omega(\bglonegradedof{E})$ has the
  homotopy type of
  $(\gloneIof{E})(\mS)$. Proposition~\ref{prop:low_dim_hty_of_bglonegraded}
  below verifies the claim about $\pi_0(\bglonegradedof{E})$.
\end{proof}

\begin{remark}
  One may also formulate Theorem~\ref{thm:bglonegraded-introduction}
  in terms of the homotopy fiber $\glonegraded{E} \simeq
  \Omega(\bglonegradedof{E})$ of the map of spectra associated with
  $\gloneJof{E} \to \bof{\cJ}$ by saying that $\glonegraded{E}$ is a
  ($-2$)-connected spectrum with $(\gloneIof{E})(\mS)$ as its
  connective cover. Another reformulation is to say that
  $\bgloneof{E}$ is the $0$-connected cover of $\bglonegradedof{E}$.
\end{remark}
We recall the notion of \emph{periodicity} from  Theorem~\ref{thm:bglonegraded-introduction}:

\begin{definition}
  The periodicity $n_E \in \mN_{0}$ of $E$ is defined to be zero if all units in
  the underlying multiplicative graded monoid of $\pi_*(E)$ have degree
  $0$. Otherwise $n_E$ is the minimal positive integer such that
  $\pi_*(E)$ has a unit of degree $n_E$.
\end{definition}

\begin{proposition}\label{prop:low_dim_hty_of_bglonegraded}
  There is an isomorphism of exact sequences
 \[
  \xymatrix@-1pc{  \pi_1(\bof{\cJ}) \ar[r] \ar[d]_{\iso}&
    \pi_1(\bglonegradedof{E}) \ar[r] \ar[d]_{\iso}& \pi_0(\gloneJof{E})
    \ar[r] \ar[d]_{\iso}&\pi_0(\bof{\cJ}) \ar[r] \ar[d]_{\iso}&
    \pi_0(\bglonegradedof{E}) \ar[r]  \ar[d]_{\iso}& 0\\
     \{\pm 1\} \ar[r] & (\pi_0E)^{\times} \ar[r] &
    (\pi_*E)^{\times}\!/\{\pm 1\} \ar[r] & \mathbb Z \ar[r] & \mathbb
    Z/n_E \mathbb Z\ar[r] & 0.  }
  \]
\end{proposition}
\begin{proof}
  By definition of the $\Gamma$-space $\gloneJof{E}$ we know
  $\pi_0(\gloneJof{E}) \iso \pi_0(\GLoneJof{E})_{h\cJ}$. The latter
  group is isomorphic to $\pi_*(E)^{\times}/\{\pm 1\}$
  by~\cite[Corollary 4.17]{Sagave-S_diagram} and~\cite[Proposition
  4.26]{Sagave-S_diagram}. The image of the map from
  $\pi_0(\gloneJof{E})$ to $\pi_0(\bof{\cJ}) \iso \pi_0(\mZ)\iso \mZ$
  is $n_E\mZ$ since it is the set of all integers that arise as the
  degree of a unit in the graded monoid $\pi_*(E)$. Hence
  $\pi_0(\bglonegradedof{E})\iso \mZ/n_E\mZ$.

  Since $(\gloneIof{E})(\mS)$ is the connective cover of
  $\Omega(\bglonegradedof{E})$, we obtain isomorphisms
  $\pi_1(\bglonegradedof{E}) \iso \pi_0(\gloneIof{E}) \iso
  \pi_0(E)^{\times}$. The action of $\pi_1(\bof{\cJ}) \iso \pi_1(\mS)
  \iso \mZ/2$ is identified with the sign action in~\cite[Proposition
  4.24]{Sagave-S_diagram}
\end{proof}

\begin{remark}The proposition implies that $\bglonegraded$ detects
  periodicity: If $E$ is periodic and $e\to E$ is its connective
  cover, $\pi_0(\bglonegraded(e \to E))$ is the surjection $\mZ \to
  \mZ/n_E\mZ$. In contrast, the map $\gloneIof{e} \to \gloneIof{E}$ is
  a stable equivalence. Its associated map of spectra is equivalent to
  the connective cover of $\Omega\bglonegraded(e \to E)$.
\end{remark}

\begin{remark}\label{rem:Hopf-map-on-bglonegraded}
  Proposition~\ref{prop:low_dim_hty_of_bglonegraded} also shows that
  the stable Hopf map (as a generator of $\pi_1(\mS)$) gives rise
  to the sign action on $(\pi_0E)^{\times}$: On $\pi_1$, the map $\mS
  \to \bglonegradedof{E}$ sends the generator to $-1 \in
  (\pi_0E)^{\times}$. This relation is not visible in the ordinary
  $\bgloneof{E}$.
\end{remark}

In most cases, $\bglonegradedof{E}$ is not just
a product of $\bgloneof{E}$ with an Eilenberg Mac-Lane spectrum on
$\mZ/n_E\mZ$:
\begin{lemma}\label{lem:first-k-inv}
  If the action by $\{\pm 1\}$ on the group $\pi_0(E)^{\times}$
  induced by the additive structure of the ring $\pi_0(E)$ is
  non-trivial, then the spectrum $\bglonegradedof{E}$ has a
  non-vanishing first $k$-invariant.
\end{lemma}
\begin{proof}
  If the first $k$-invariant of $\bglonegradedof{E}$ vanishes, the
  Postnikov section $P_1(\bglonegradedof{E})$ decomposes as the
  product of suspended Eilenberg-Mac Lane spectra. In view of
  Proposition~\ref{prop:low_dim_hty_of_bglonegraded}, this is a
  contradiction because the map from $\mS$ to
  $P_1(\bglonegradedof{E})$ that is the surjection $\mZ \to
  \mZ/n_E\mZ$ on $\pi_0$ is non-trivial on $\pi_1$ by assumption.
\end{proof}

\begin{example} The last Lemma applies to the periodic complex
  $K$-theory spectrum~$KU$. We have $\pi_0\bglonegradedof{KU} \iso
  \mZ/2\mZ$ and $\pi_1\bglonegradedof{KU} \iso \pi_0(KU)^{\times} \iso
  \mZ/2\mZ$. So the first $k$-invariant of $\bglonegradedof{KU}$ is a
  non-zero element of $H^2(H\mZ/2; \mZ/2)$ and hence
  equals~$\text{Sq}^2$.
\end{example}

\begin{remark}
  We give a brief sketch of how the map $\mS \to \bglonegradedof{E}$
  resulting from the construction of $\bglonegradedof{E}$ can be
  understood in terms of more well known maps. A detailed account of
  this will be included in~\cite{Sagave-S_Thom-graded-units}.
  Naturality of $\gloneJ$ implies that $\mS \to \bglonegradedof{E}$
  factors as the composite of $\mS \to \bglonegradedof{\,\mS}$ and the
  map $\bglonegradedof{\,\mS} \to \bglonegradedof{E}$ induced by $\mS
  \to E$. So we may restrict to the case $E=\mS$. 

  Let $\mathcal W = \mathcal{O}^{-1}\mathcal{O}$ be Quillen's
  localization construction on the category $\cO$ of standard inner
  product spaces $\mathbb R^n$ and linear isometric isomorphisms.  The
  category $\cW$ is the orthogonal counterpart of $\cJ$ and is studied
  in~\cite{Sagave-S_Virtual-vector}.  It is related to orthogonal
  spectra in the same way as $\cJ$ is related to symmetric spectra,
  and to the orthogonal version of $\cI$-spaces studied by
  Lind~\cite{Lind-diagram} in the same way as $\cJ$ is related to
  $\cI$.  (The $\cW$-spaces considered here and
  in~\cite{Sagave-S_Virtual-vector} should not be confused with the
  continuous functors on based spaces homeomorphic to finite
  CW-complexes studied under this name elsewhere,
  e.g. in~\cite[Example 4.6]{MMSS}.)

   Since $\mS$ can be viewed as an orthogonal spectrum, it has
  an associated commutative $\mathcal W$-space monoid of units and an
  associated $\Gamma$-space $\gl^{\cW}_1\!\!{\,\mS}$ augmented over
  $\bof{\mathcal W}$. The cofiber of the map of spectra induced by the
  augmentation is equivalent to $\bglonegradedof{E}$, an it follows
  that $\mS \to \bglonegradedof{\,\mS}$ factors as the map $\mS \simeq
  \bof{\cJ} \to \bof{\mathcal W} \simeq \mathrm{ko}$ induced by the
  unit and a map $\mathrm{ko} \simeq \bof{\mathcal W} \to
  \bglonegradedof{\,\mS}$. The underlying $E_{\infty}$ space of the
  $0$-connected cover of $\mathrm{ko}$ is $BO$, the underlying
  $E_{\infty}$ space of the $0$-connected cover $\bgloneof{\,\mS}$ of
  $\bglonegradedof{\,\mS}$ is $BF$, and the map of underlying
  $E_{\infty}$ spaces of the $0$-connected covers is the stable
  $J$-homomorphism $BO \to BF$.
\end{remark}

\section{Grouplike graded commutative spaces}\label{sec:group-compl-model-str}
In Section~\ref{sec:Gamma-spaces-from-monoids} we observed that a
grouplike commutative $\cJ$-space monoid $A$ gives rise to a map of
very special $\Gamma$-spaces $\gamma(A) \to \bof{\cJ}$. In this
section we explain how this construction induces an equivalence
between the homotopy categories of grouplike commutative $\cJ$-space
monoids and $\Gamma$-spaces over the $\Gamma$-space $\bof{\cJ}$ and
why this is relevant for the graded units $\gloneJ$.

We start by recalling the definition of the \emph{positive $\cJ$-model
  structure} on commutative $\cJ$-space monoids from~\cite[\S
4]{Sagave-S_diagram}. The weak equivalences are the $\cJ$-equivalences
of Definition~\ref{def:K-equivalences}, that is, the maps which induce
weak equivalences on homotopy colimits over $\cJ$.  A map $f\colon A
\to B$ in $\cC\cS^{\cJ}$ is a \emph{positive $\cJ$-fibration} if every
morphism $(\beta_1, \beta_2, \sigma)\colon (\bld{m_1},\bld{m_2})\to
(\bld{n_1},\bld{n_2})$ in $\cJ$ with $m_1 \geq 1$ induces a homotopy
cartesian square
\[
\xymatrix@-1pc{
  A(\bld{m_1},\bld{m_2}) \ar[d] \ar[r] & A(\bld{n_1},\bld{n_2}) \ar[d] \\
  B(\bld{m_1},\bld{m_2}) \ar[r]& B(\bld{n_1},\bld{n_2}) 
}
\]
in which the vertical maps are Kan fibrations.  The \emph{positive
  $\cJ$-cofibrations} are determined by a lifting property. It is
shown in~\cite[Proposition 4.10]{Sagave-S_diagram} that these classes
of maps form a (cofibrantly generated proper) model structure on
$\cC\cS^{\cJ}$. The analogy to $\cI$-spaces discussed in the introduction and the following
two results indicate that this model structure gives an interesting homotopy 
theory of commutative $\cJ$-space monoids. 
\begin{proposition}[{\cite[Proposition 4.23]{Sagave-S_diagram}}]\label{prop:OmegaJ-right-Quillen}
  There is a Quillen adjunction
  \[ \mS^{\cJ}[-] \colon \cC\cS^{\cJ}\rightleftarrows
  \cC\Spsym\colon\OmegaJ\] with respect to the positive $\cJ$-model
  structure and the positive stable model structure on the category of
  commutative symmetric ring spectra introduced in~\cite{MMSS}.
\end{proposition}
\begin{theorem}[{\cite[Theorem 1.7]{Sagave-S_diagram}}]
  There is a chain of Quillen equivalences between the positive
  $\cJ$-model structure on $\cC\cS^{\cJ}$ and the category of
  $E_{\infty}$ spaces over $B\cJ$.
\end{theorem}
In the theorem, $E_{\infty}$ spaces are spaces with an action of the
Barratt-Eccles operad. It uses that the latter acts on the classifying
space of a permutative category.

\subsection{Group completion for commutative
  \texorpdfstring{$\cJ$}{J}-space monoids}\label{subsec:group-completion}
The \emph{recognition principle} for infinite loop spaces states that
grouplike $E_{\infty}$ spaces are equivalent to connective
spectra~\cite{Boardman-V_homotopy-invariant,May_geometry}. Together
with the last theorem, this suggests that grouplike commutative
$\cJ$-space monoids should be equivalent to connective spectra over the
spectrum associated with the permutative category $\cJ$. 

We will prove this by first building a model category that is closely
related to group completion and whose homotopy category is equivalent
to the subcategory of grouplike objects in $\Ho(\cC\cS^{\cJ})$. In a
second step, we show that there is a chain of Quillen equivalences
relating this model category to the category of $\Gamma$-spaces over
$\bof{\cJ}$ with a stable model structure.  This implies the above
statement about the equivalence of the respective homotopy
categories. At the same time, it shows the much stronger statement
that homotopy classes of maps and the homotopy types of mapping spaces
coincide.

A similar description of the classical recognition principle in terms
of Quillen equivalences is given in~\cite{Sagave-S_group-compl}.
Working in the context of commutative $\cI$-space monoids, it is even
possible to build a \emph{single} Quillen equivalence inducing an
equivalence between the homotopy category of grouplike commutative
$\cI$-space monoids and the homotopy category of connective spectra.

In the following we use that the homotopy colimit $A_{h\cJ}$
associated with a commutative $\cJ$-space monoid is a simplicial
monoid so that we can form the bar construction
$B(A_{h\cJ})=B(*,A_{h\cJ},*)$.
\begin{definition}\label{def:group-completion} \begin{enumerate}[(i)]
\item   A commutative $\cJ$-space monoid $A$ is \emph{grouplike} if
  the commutative monoid $\pi_0(A_{h\cJ})$ is a group.
\item A map of commutative $\cJ$-space
  monoids $A \to A'$ is a \emph{group completion} if
  $B(A_{h\cJ}) \to B(A'_{h\cJ})$ is a weak equivalence of spaces
  and $A'$ is grouplike.
\item A map $A \to A'$ in $\cC\cS^{\cJ}$ induces an \emph{equivalence
    after group completion} if the induced map of bar constructions
  $B(A_{h\cJ}) \to B(A'_{h\cJ})$ is a weak equivalence.
\end{enumerate}
\end{definition}
Following the usual terminology that map of $\Gamma$-spaces is
a \emph{stable equivalence} if it induces a stable equivalence of
associated spectra, Proposition~\ref{prop:gamma-A-special} implies
that $A \to A'$ induces an equivalence after group completion if and
only if $\gamma(A) \to \gamma(A')$ is a stable equivalence. We note
that the existence of group completions in $\cC\cS^{\cJ}$ is not
needed for this notion of equivalence.

\begin{theorem}\label{thm:group-completion-model-str}
  The category of commutative $\cJ$-space monoids $\cC\cS^{\cJ}$
  admits a model structure in which the weak equivalences are the maps
  that induce weak equivalences after group completion and the
  cofibrations are the positive $\cJ$-cofibrations.

  A commutative $\cJ$-space monoid is fibrant in this model structure
  if and only if it is positive $\cJ$-fibrant and grouplike. The
  fibrant replacement is a group completion in the sense of
  Definition~\ref{def:group-completion}. Fibrations are
  characterized by the right lifting property.
\end{theorem}
We refer to this model structure as the \emph{group completion model
  structure} and write $\cC\cS^{\cJ}_{\grp}$ for this model category.
Its homotopy category is the homotopy category of grouplike
commutative $\cJ$-space monoids.

The model structure of Theorem 5.5 will be constructed by localizing
the positive $\cJ$-model structure with respect to an explicit set of
``shear maps'' to be defined in~\eqref{eq:shear-maps-for-CSJ}
below. The difficult part is to identify the weak equivalences and
fibrant objects. We prove the theorem in
Section~\ref{subsec:stable-model-Gamma} below.
\begin{remark}
  A similar group completion model structure for commutative
  $\cI$-space monoids has been developed
  in~\cite{Sagave-S_group-compl}.  A key ingredient for the
  $\cI$-space result was the fact that the ordinary group completion
  of simplicial monoids lifts to commutative $\cI$-space monoids: One
  can define a bar construction $B^{\boxtimes}$ and a loop functor
  $\Omega$ so that $A \to \Omega B^{\boxtimes}(A)$ is a group
  completion for the commutative $\cI$-space monoid $A$. This does not
  carry over to commutative $\cJ$-space monoids: Since $\cC\cS^{\cJ}$
  has no zero object, there is no loop functor on $\cC\cS^{\cJ}$ that
  serves for the above purpose.  (The point is that already for
  ordinary $\mZ$-graded commutative monoids, the initial object is
  concentrated in degree $0$ while the terminal object is non-empty in
  every degree.)

  In particular, the fibrant replacement in the above model structure
  is the only construction of a group completion functor for
  commutative $\cJ$-space monoids we are currently aware of.
\end{remark}

\begin{remark}
  As discussed in Section~\ref{subsec:log-applications},
  Theorem~\ref{thm:group-completion-model-str} is used in an essential
  way for the study of ring spectra with graded logarithmic structures
  and their logarithmic topological Andr\'{e}-Quillen homology and
  logarithmic topological Hochschild homology. We refer
  to~\cite{Sagave_log-on-k-theory, RSS_LogTHH-I, RSS_LogTHH-II} for details
  of these applications.
\end{remark}

\begin{example}
  We analyze the group completion of commutative $\cJ$-space monoids
  in an example. Let
  \begin{equation} \label{eq:free-com-J} A = \mC
    F^{\cJ}_{(\bld{m_1},\bld{m_2})}(*) \iso \textstyle\coprod_{n\geq
      0} (F^{\cJ}_{(\bld{m_1},\bld{m_2})}(*))^{\boxtimes
      n}/\Sigma_n \end{equation} be the free commutative $\cJ$-space
  monoid on a point in degree $(\bld{m_1},\bld{m_2})$. We assume $m_1
  \geq 1$ to ensure that the symmetric group action in the definition
  of $A$ is free. It follows that
  \[ A_{h\cJ} \simeq \textstyle\coprod_{n\geq 0} B\Sigma_n.\] Let $A
  \to A^{\grp}$ be a fibrant replacement in the group completion model
  structure. The Barratt-Priddy-Quillen theorem implies that $A_{h\cJ}
  \to (A^{\grp})_{h\cJ}$ has the homotopy type of the group completion
  map $\textstyle\coprod_{n\geq 0} B\Sigma_n \to QS^0$. 

  The map $A^{\grp} \to *$ induces an augmentation $QS^0 \simeq
  (A^{\grp})_{h\cJ} \to B\cJ \simeq QS^0$. The homotopy class of this
  map is determined by the image of a generator under the induced map
  of path components. Since the generator of $\pi_0(A_{h\cJ}) \iso
  \mN$ is mapped to $m_2-m_1$ in $\pi_0(B\cJ)\iso \mZ$, it follows
  that $(A^{\grp})_{h\cJ} \to B\cJ$ is multiplication by
  $m_2-m_1$. This implies that in the special case $m_2-m_1=1$, the
  group completion $A^{\grp}$ is particularly simple: The map $
  A^{\grp}\to *$ is a $\cJ$-equivalence.
\end{example}

\subsection{Grouplike commutative \texorpdfstring{$\cJ$}{J}-space
  monoids and connective spectra}
In order to give the characterization of grouplike commutative
$\cJ$-space monoids outlined above, we will again use $\Gamma$-spaces as a
model for connective spectra. The category $\GammaS$ admits a
\emph{stable Quillen model structure} (stable Q-model structure for
short) whose homotopy category is the homotopy category of connective
spectra~\cite{Schwede_Gamma-spaces,
  Bousfield-F_Gamma-bisimplicial}. We write $(\GammaS)_{\sta}$ for
this model category. It will be reviewed in
Section~\ref{subsec:stable-model-Gamma}.

We will use that for an object $Z$ in a model category $\cM$, the
comma category $\cM\downarrow Z$ of objects over $Z$ inherits a
``overcategory'' model structure in which a map is a weak equivalence,
cofibration or fibration if its projection to $\cM$ is~\cite[\S
7.6]{Hirschhorn_model}.

\begin{theorem}\label{thm:identification-of-group-compl-model-structure}
  There is a chain of Quillen equivalences 
  relating $\cC\cS^{\cJ}$ with the group completion model structure
  and the category $(\GammaS)_{\sta}/\bof{\cJ}$. 

  The chain of Quillen equivalences can be chosen so that the value of
  the composed derived functor $\cC\cS^{\cJ} \to (\GammaS)/\bof{\cJ}$
  on a commutative $\cJ$-space monoid $A$ and the explicit
  $\Gamma$-space $\gamma(A)$ of
  Definition~\ref{def:Gamma-space-from-K-space-monoid} are stably
  equivalent over $\bof{\cJ}$.
\end{theorem}
Here the \emph{derived functor} of a left (or right) Quillen
functor means the functor of model categories obtained by precomposing
it with a cofibrant (or fibrant) replacement functor. The
\emph{composed derived functor} is the composite of the derived
functors associated with the Quillen functors in the zig-zag of
Quillen equivalences.
 
The proof of
Theorem~\ref{thm:identification-of-group-compl-model-structure} and
the construction of the intermediate model categories that are
implicit in its formulation will be given in
Section~\ref{sec:stable-model-str}.
Theorem~\ref{thm:identification-of-group-compl-model-structure} and
Theorem~\ref{thm:group-completion-model-str} combine to give
Theorem~\ref{thm:gp-model-str-and-comparison-introduction} from the
introduction.

\subsection{Graded units of ring spectra as a right adjoint}
We earlier defined the units $A^{\times}$ of a commutative
$\cJ$-space monoid $A$ as the sub commutative $\cJ$-space monoid of
invertible path components.  They come with a natural map $A^{\times}
\to A$. This construction admits a useful formulation in terms of a
model structure which is somewhat dual to the group completions
discussed in Section~\ref{subsec:group-completion}:
\begin{theorem}\label{thm:units-model-structure}
  The category of commutative $\cJ$-space monoids $\cC\cS^{\cJ}$
  admits a model structure in which a map $f\colon A \to B$ is a weak
  equivalences if the induced map $f^{\times}\colon A^{\times}\to
  B^{\times}$ is a $\cJ$-equivalence. The fibrations are the positive
  $\cJ$-fibrations. 

  The cofibrant objects are the commutative $\cJ$-space monoids which
  are positive $\cJ$-cofibrant and grouplike, and the cofibrant
  replacement of $A$ is $\cJ$-equivalent to $A^{\times} \to A$.  The
  general cofibrations are determined by a lifting property.\qed
\end{theorem}
We call the resulting model structure $\cC\cS^{\cJ}_{\mathrm{un}}$ the
\emph{units} model structure.  The proof of the theorem is completely
analogous to the proof of the corresponding statement about
commutative $\cI$-space monoids in~\cite[Theorem
1.8]{Sagave-S_group-compl} and will not be repeated here.

The identity functors form a Quillen adjunction
$\cC\cS^{\cJ}_{\mathrm{un}} \rightleftarrows \cC\cS^{\cJ}$ (where as
always in this paper the left adjoint is the upper arrow). Together
with the Quillen adjunction $\cC\cS^{\cJ} \rightleftarrows
\cC\cS^{\cJ}_{\grp}$ coming from the group completion model structure,
we obtain
\begin{corollary}
  The identity functors form a Quillen equivalence
  $\cC\cS^{\cJ}_{\mathrm{un}} \rightleftarrows \cC\cS^{\cJ}_{\grp}$.
\end{corollary}
\begin{proof}
  On both sides, a weak equivalence from a
  positive $\cJ$-cofibrant and grouplike object into a positive
  $\cJ$-fibrant and grouplike object is a $\cJ$-equivalence.
\end{proof}
Combining this with the Quillen adjunction of
Proposition~\ref{prop:OmegaJ-right-Quillen} we
get a chain of Quillen adjunctions 
\begin{equation}\label{eq:chain-adjunctions-glone-right-adj}
  \cC\Spsym \leftrightarrows \cC\cS^{\cJ} \leftrightarrows 
  \cC\cS^{\cJ}_{\mathrm{un}}\rightleftarrows \cC\cS^{\cJ}_{\grp}
\end{equation}
in which the last one is a Quillen equivalence. We can use this to
identify $\gloneJ$ as right adjoint on the level of homotopy
categories:
\begin{proof}[Proof of Theorem~\ref{thm:sp-of-graded-units-as-righ-adj-introduction}]
  For a positive fibrant commutative symmetric ring spectrum,
  Theorem~\ref{thm:units-model-structure} implies that the units
  $\GLoneJof{E}$ are $\cJ$-equivalent to the cofibrant replacement of
  $\OmegaJof{E}$ in $\cC\cS^{\cJ}_{\mathrm{un}}$. Together with
  Theorem~\ref{thm:identification-of-group-compl-model-structure}, the
  chain of adjunctions~\eqref{eq:chain-adjunctions-glone-right-adj}
  this implies that $\gloneJof{E} = \gamma(\GLoneJof{E})$ is as an
  augmented object stably equivalent to the composed derived functor
  $\cC\Spsym \to \left(\GammaS\right) / \bof{\cJ}$ of a zig-zag of
  Quillen adjunctions.

  Right Quillen functors induce right adjoints of homotopy
  categories. Left Quillen functors that participate in Quillen
  equivalences induce equivalences of homotopy categories, which are
  in particular right adjoints. Since all left Quillen functors in the
  zig-zag of Quillen adjunctions describing $\gloneJ$ are part of a
  Quillen equivalence, the claim follows.
\end{proof}
\begin{remark}
  Since the equivalences of homotopy categories arise from inverting
  Quillen equivalences, the previous proof shows more than is stated in
  the theorem: The adjunction $\Ho(\cC\Spsym ) \leftrightarrows
  \Ho(\left(\GammaS\right)_{\sta} / \bof{\cJ})$ induced by $\gloneJ$
  is compatible with the homotopy types of mapping spaces on both
  sides.
\end{remark}
\section{Gamma-\texorpdfstring{$\cJ$}{J}-spaces}\label{sec:GammaJ-spaces}
In this section we construct the zig-zag of adjunctions that will be
used in the proofs of Theorem~\ref{thm:group-completion-model-str} and
Theorem~\ref{thm:identification-of-group-compl-model-structure}, and
we make the first step towards the model structures that appear in the
theorems.

Let $F\colon \cC \to \text{CAT}$ be a functor to the category of small
categories. Its Grothendieck construction is the category $\cC\int F$
whose objects are pairs $(C,X)$ with $C \in \obj(\cC)$ and $X \in
\obj(F(C))$. A morphism $(\alpha, f)\colon (C,X) \to (D,Y)$ consists
of a morphism $\alpha \colon C \to D$ in $\cC$ and a morphism $f\colon
F(\alpha)(X)\to Y$ in $F(D)$. Its composite with $(\beta,g)\colon
(D,Y) \to (E,Z)$ is $(\beta\alpha, g(F(\beta)(f)))$.

The next definition refers to the functor $\HcJ\colon \Gamma^{op} \to
\text{CAT}$ defined in Section~\ref{subsec:delooping-perm}.
\begin{definition}
  We let $\GammaJ$ be the Grothendieck construction $\Gamma^{\op}\int
  \HcJ$ on $\HcJ$.
\end{definition}
Objects of $\GammaJ$ are tuples $(S; \bld{s},\sigma)$ with $S \in
\obj(\Gamma^{\op})$ a finite based set and $(\bld{s},\sigma)$ an
object of the category $\HcJ(S)$. The object $(0^+;\bld{0}_{0^+})$ is
terminal in $\GammaJ$. (It is not initial because not every object in
$\cJ$ receives a map from $(\bld{0},\bld{0})$.) 
\begin{definition}
  Let $\GammaJS$ be the category of functors $X\colon \GammaJ \to \cS$
  sending $(0^+;\bld{0}_{0^+})$ to the one point space.
\end{definition}
\begin{example}
  For $(\bld{s},\sigma) \in \obj(\HcJ(S))$, the functor
  $\GammaJ((S;\bld{s},\sigma), -)\colon \GammaJ \to \cS$ defines a
  $\GammaJ$-space since $(0^+,\bld{0}_{0^+})$ is terminal.
\end{example}
The category $\GammaJS$ admits an equivalent description that will provide
us with more examples.  For this we
recall from Section~\ref{subsec:Gamma-from-CSK} that $\alpha \colon S
\to T$ in $\Gamma^{\op}$ induces a functor $\alpha_*\colon
\HcJ(S)\to\HcJ(T)$ and hence $\alpha^*\colon
\cS^{\HcJ(T)}\to\cS^{\HcJ(S)}$.
\begin{lemma}\label{lem:GammaJS-equivalent}
  A $\GammaJ$-space is the same as a collection of $\HcJ(S)$-spaces
  $X_S$ for every finite based set $S$ and maps of $\HcJ(S)$-spaces
  $\widetilde{\alpha}\colon X_S \to \alpha^*(X_T)$ for every map
  $\alpha\colon S \to T$ in~$\Gamma^{\op}$ such that $X_{0^+}=*$ and
  $(\alpha^*(\widetilde{\beta})) \widetilde{\alpha} =
  \widetilde{\beta\alpha}$ holds for composable maps in
  $\Gamma^{\op}$. Morphisms of $\GammaJ$-spaces correspond to families
  of maps of $\HcJ(S)$-spaces $f_S \colon X_S \to Y_S$ such that
  $(\alpha^{*}(f_T))\widetilde{\alpha} = \widetilde{\alpha}f_S$.
\end{lemma}
\begin{proof}
If $X$ is a $\GammaJ$-space, then the $X_S = X(S;-)$ have this property. The other
direction follows from the definition of the Grothendieck construction. 
\end{proof}
The lemma applies to the $\HcJ(S)$-spaces $A_S$ and maps
$\widetilde{\alpha}$ defined in Section~\ref{subsec:Gamma-from-CSK}:
\begin{corollary}\label{cor:GammaJ-sp-from-CSJ}
  The collection of $\HcJ(S)$-spaces $A_S$ associated with a
  commutative $\cJ$-space monoid $A$ defines a $\GammaJ$-space.\qed
\end{corollary}

Colimits in $\GammaJS$ are not the levelwise colimits of the
underlying $\HcJ(S)$-spaces because this would violate the condition
on $(0^+,\bld{0}_{0^+})$. To construct them, we note that the
basepoint condition on a $\GammaJ$-space $X$ induces a canonical map
$\GammaJ((0^+;\bld{0}_{0^+}),-)\to X$. Hence $\GammaJS$ may be viewed
as a full subcategory of
$\GammaJ((0^+,\bld{0_{0^+}}),-)\!\downarrow\!\cS^{\GammaJ}$, and
the colimits and limits in the larger category exist and provide
colimits and limits in $\GammaJS$. This proves
\begin{lemma}\label{lem:GammaJS-complete-cocomplete}
The category $\GammaJS$ is complete and cocomplete.\qed
\end{lemma}
Let $X$ and $Y$ be $\GammaJ$-spaces and $K$ a simplicial set. We
define $X\tensor K$ to be the $\GammaJ$-space given by the pushout
\[ \GammaJ((0^+,\bld{0}_{0^+}),-) \ot
\GammaJ((0^+,\bld{0}_{0^+}),-)\times K \to X \times K \] in
$\cS^{\GammaJ}$. Here $-\times K$ denotes the objectwise product with
$K$. Moreover, we define a $\GammaJ$-space $X^K$ by
$X^K(S;\bld{s},\sigma) = (X(S;\bld{s},\sigma))^K$ and a simplicial set
$\Map(X,Y)$ by $[m]\mapsto \GammaJS(X\tensor \Delta^m,Y)$. Using
\cite[II Lemma 2.4]{Goerss-J_simplicial} we conclude
\begin{proposition}\label{prop:GammaJS-tens-cotens-enr}
  With these definitions, $\GammaJS$ is tensored, cotensored, and
  enriched over unpointed simplicial sets. \qed
\end{proposition}

For the construction of model structures on $\GammaJS$ we will need
certain free functors that we introduce next. Let $X$ be a
$\GammaJ$-space and let $X_S$ be the associated $\HcJ(S)$-space. The
unique map $\alpha\colon 0^+ \to S$ induces a map $* =
X_{0^+}(\bld{0}_{0^+}) \to \alpha^*(X)(\bld{0}_{0^+}) =
X_S(\bld{0}_S)$ that makes $X_S(\bld{0}_S)$ a pointed space and $X_S$
an $\HcJ(S)$-space under the free $\HcJ(S)$-space
$\HcJ(S)(\bld{0}_S,-)$. By slight abuse of notation, we write
$\cS^{\HcJ(S)}_*$ for the category $\HcJ(S)(\bld{0}_S,-) \downarrow
\cS^{\HcJ(S)}$ and call its objects \emph{pointed
  $\HcJ(S)$-spaces}. (These are not $\HcJ(S)$-diagrams in pointed
spaces!) The proof of Lemma~\ref{lem:GammaJS-complete-cocomplete}
implies
\begin{corollary}\label{cor:colimits-in-GammaJS}
The evaluation $\GammaJS \to \cS^{\HcJ(S)}_*$ preserves colimits.\qed
\end{corollary}
The functor $\alpha_*\colon \HcJ(S)\to\HcJ(T)$
induced by $\alpha\colon S \to T$ in $\Gamma^{\op}$ gives rise to an
adjunction
\begin{equation}\label{eq:pointed-adj-alpha}
\alpha_{!}\colon \cS^{\HcJ(S)}_* \rightleftarrows \cS^{\HcJ(T)}_*\colon\alpha^*.
\end{equation}
The right adjoint is precomposition, and
$\alpha_*(\bld{0}_S)=\bld{0}_T$ ensures that $\alpha^*(Y)(\bld{0}_S)$
is pointed. The left adjoint is defined by a left Kan extension.
\begin{lemma}\label{lem:free-functor-FS}
  The evaluation functor $\GammaJS \to \cS^{\HcJ(S)}_*, \; X \mapsto X(S;-)$ has a left
  adjoint $F_S$. For a pointed $\HcJ(S)$-space $Z$, there is an
  isomorphism
  \begin{equation}\label{eq:free-on-ptd-hjS-space}
    (F_SZ)_T \iso \textstyle\coprod_{\alpha\colon S \to T} \alpha_{!}(Z) 
  \end{equation}
  where $\alpha$ ranges over all non-zero maps in $\Gamma^{\op}$ and
  the coproduct is taken in~$\cS^{\HcJ(T)}_*$.
\end{lemma}
\begin{proof}
  We apply Lemma~\ref{lem:GammaJS-equivalent} to see that the right
  hand site in~\eqref{eq:free-on-ptd-hjS-space} defines a
  $\GammaJ$-space. By~\eqref{eq:pointed-adj-alpha} it is enough that
  there are maps $\widehat{\beta}\colon\beta_{!}(F_SZ)_T \to (F_SZ)_U$
  for every $\beta \colon T \to U$ in $\Gamma^{\op}$ such that
  $\widehat{\gamma}(\gamma_!\widehat{\beta})=\widehat{\gamma\beta}$
  holds for composable maps. On the summand indexed by $\alpha\colon S
  \to T$, we define $\widehat{\beta}$ to be the inclusion of the
  summand $(\beta\alpha)_!(Z)$ if $\beta\alpha$ is non-zero. If
  $\beta\alpha$ factors as the composite $\nu \varepsilon$ with
  $\nu\colon 0^+\to U$, then $(\beta\alpha)_!(Z) \iso
  \HcJ(U)(\bld{0}_U,-)\times \varepsilon_!(Z)$, and we define
  $\beta_!\alpha_!(Z) \to (F_SZ)_U$ to be the map induced by the
  collapse of $\varepsilon_!(Z)$ and the inclusion of the
  basepoint. It is easy to check that the $\widehat{\beta}$ satisfy
  the required compatibility and that this defines the desired left
  adjoint $F_S$.
\end{proof} 
Using various free-forgetful adjunctions, free $\GammaJ$-spaces on
free pointed $\HcJ(S)$-spaces can also be expressed using the tensor
introduced in Proposition~\ref{prop:GammaJS-tens-cotens-enr}:
\begin{corollary}\label{cor:tensor-and-free}
  For objects $(\bld{s},\sigma)$ of $\HcJ(S)$, there
  is a natural isomorphism \[F_S\left(F_{\bld{0}_S}^{\HcJ(S)}(*)
    \coprod F_{(\bld{s},\sigma)}^{\HcJ(S)}(K)\right) \iso
  \GammaJ((S;\bld{s},\sigma),-)\tensor K.\]\qed
\end{corollary}
\subsection{Commutative \texorpdfstring{$\cJ$}{J}-space monoids and
  \texorpdfstring{$\GammaJ$}{Gamma-J}-spaces}
We give an alternative description of the passage from a commutative
$\cJ$-space monoid to a $\GammaJ$-space provided by
Corollary~\ref{cor:GammaJ-sp-from-CSJ} in order to see that this
construction is a right adjoint.

The free commutative $\cJ$-space monoid $\mC
F_{(\bld{n_1},\bld{n_2})}^{\cJ}(*)$ on a point in degree
$(\bld{n_1},\bld{n_2})$ is the image of the free $\cJ$-space
$F_{(\bld{n_1},\bld{n_2})}^{\cJ}(*) = \cJ((\bld{n_1},\bld{n_2}), -)$
under the free functor $\mC\colon \cS^{\cJ} \to \cC\cS^{\cJ}$
(compare~\eqref{eq:free-com-J}).  It is contravariant functorial in
$(\bld{n_1},\bld{n_2})$, and using the coproduct in the category
$\cC\cS^{\cJ}$ we obtain functors
\begin{equation}\label{eq:functor-CFS} \mC F_S \colon \HcJ(S)^{\op}
  \to \cC\cS^{\cJ}, \qquad (\bld{s},\sigma)\mapsto \textstyle\coprod_{i\in
    \ovl{S}} \mC F_{\bld{s}_i}^{\cJ}(*). 
\end{equation}
Since the sum in $\cC\cS^{\cJ}$ is the $\boxtimes$-product, a
choice of an ordering of $\ovl{S}$ induces an isomorphism between
$(\mC F_S)(\bld{s},\sigma)$ and the iterated $\boxtimes$-product of
the commutative $\cJ$-space monoids $\mC F_{\bld{s}_i}^{\cJ}(*)$.

The functors $\mC F_S$ for varying $S$ are related: A morphism $\alpha
\colon S \to T$ in $\Gamma^{\op}$ induces a natural transformation
$\alpha^* \colon \mC F_T \circ \alpha_* \to \mC F_S$ of functors
$\HcJ(S)^{\op}\to\cC\cS^{\cJ}$ such that for a second morphism
in $\beta\colon T\to U$ in $\Gamma^{\op}$, the composite
\[ \mC F_U \circ \beta_* \circ \alpha_* \to \mC F_T \circ \alpha_* \to
\mC F_S \] equals $(\beta\alpha)^* \colon \mC F_U \circ (\beta
\alpha)_* \to \mC F_S$. The definition of $\alpha^*$ is similar to
Section~\ref{subsec:Gamma-from-CSK}: Choosing an ordering
$\{i_1,\dots,i_v\}$ of $V = \alpha^{-1}(j)$ for $j\in\ovl{T}$, we have
to define a map
\[ \mC F_{\bld{s}_V}^{\cJ}(*) \to \mC
F_{\bld{s}_{i_1}}^{\cJ}(*)\boxtimes \dots \boxtimes \mC
F_{\bld{s}_{i_v}}^{\cJ}(*).\] It is equivalent to specify a point in
the evaluation of the codomain at $\bld{s}_V$. The canonical points in
$(\mC F_{\bld{s}_{i}}^{\cJ}(*))(\bld{s_i})$ together with the morphism
$s_{i_1} \concat \dots \concat s_{i_v} \to s_{V}$ induced by the
chosen ordering and the isomorphisms that are part of the data of
$(\bld{s},\sigma)$ define such a point. The resulting map does not
depend on the choices and satisfies the above naturality.

By the definition of the Grothendieck construction we obtain
\begin{corollary}
  The functors $\mC F_S$ induce a functor $\mC F\colon (\GammaJ)^{op}
  \to \cC\cS^{\cJ}$ with $(\mC F)(S;\bld{s},\sigma) = \left(\mC F_S)(\bld{s},\sigma\right)$.\qed
\end{corollary}
The category of commutative $\cJ$-space monoids is tensored,
cotensored, and enriched over unpointed simplicial sets. Tensor and
mapping space are defined by
\[ A \tensor K = \left|[m]\mapsto A^{\boxtimes K_m}\right| \quad\text{
  and }\quad \Map(A,B) = ([m]\mapsto \cC\cS^{\cJ}(A\tensor
\Delta^m,B))\] where $|-|$ denotes the realization of simplicial
objects defined in terms of the diagonal. The cotensor is defined on
the underlying $\cJ$-spaces. Together with this structure, the
positive $\cJ$-model structure on $\cC\cS^{\cJ}$ is a simplicial model
category~\cite[Definition 9.1.6]{Hirschhorn_model} since the positive
$\cJ$-model structure on $\cJ$-space is simplicial~\cite[Proposition
6.19]{Sagave-S_diagram} and the compatibility with the model structure
can be checked on the cotensor.

\begin{definition}
  For a commutative $\cJ$-space monoid $A$, we let $\Psi(A)$ be the
  $\GammaJ$-space $\Map_{\cC\cS^{\cJ}}(\mC F (-), A)$. For a
  $\GammaJ$-space $X$, we let $\Phi(X)$ be the coend over $\GammaJ$ of
  the functor $ \mC F \tensor X$.
\end{definition}
The adjunction of mapping space and tensor with a space imply 
\begin{corollary}\label{cor:Phi-Psi-adj}
The functors
$\Phi$ and $\Psi$ are the left and right adjoints in an adjunction 
$\Phi\colon\GammaJS  \rightleftarrows \cC\cS^{\cJ}\colon\Psi$.
\qed
\end{corollary}
Since $\mC F_s$ is defined as a coproduct of free commutative
$\cJ$-space monoids, there is an isomorphism $\Map_{\cC\cS^{\cJ}}(\mC
F^{\cJ}_{S}(*), A) \iso A_S$. Comparing the definition of the
structure maps, we get an alternative description of the
$\GammaJ$-space of Corollary~\ref{cor:GammaJ-sp-from-CSJ}:
\begin{corollary}\label{cor:identification-of-Psi}
  Let $A$ be a commutative $\cJ$-space monoid. The $\GammaJ$-space
  associated with the $\HcJ(S)$-spaces $A_S$ is isomorphic to
  $\Psi(A)$.\qed
\end{corollary}
Generalizing the functor $\gamma\colon \cC\cS^{\cJ} \to \GammaS$
introduced in Definition~\ref{def:Gamma-space-from-K-space-monoid}, we
obtain a functor $\gamma'\colon \GammaJS \to \GammaS$ by setting
$\gamma'(X)(S) = X(S;-)_{h\HcJ(S)}$.
\begin{corollary}\label{cor:gamma-vs-gamma-prime}
  There is a natural isomorphism $\gamma'(\Psi(A)) \iso
  \gamma(A)$.\qed
\end{corollary}

\begin{remark}
  The usual bar construction $BA = B(*,A,*)$ in $\cC\cS^{\cJ}$ seems
  to be of limited use for our purposes since the terminal object of
  $\cC\cS^{\cJ}$ is not initial. (This means in particular that we
  cannot form $\Omega(BA)$.)

  However, one may interpret $\Psi$ as a generalized bar construction:
  The $\Psi(A)$ encodes iterated products of the spaces
  $A(\bld{m_1},\bld{m_2})$ and the structure maps between them induced
  by the multiplication and unit of $A$, without ever attempting to
  realize this as an object of $\cC\cS^{\cJ}$.
  Corollary~\ref{cor:Phi-Psi-adj} justifies why the more elaborate
  category $\GammaJS$ is an appropriate codomain of such a generalized
  bar construction.
\end{remark}

\subsection{Level model structures on
  \texorpdfstring{$\GammaJ$}{Gamma-J}-spaces}
Standard model category arguments (e.g.~\cite[\S
11.6]{Hirschhorn_model}) show that the category $\GammaS$ admits a
cofibrantly generated proper level model structure in which a map $X
\to Y$ is a weak equivalence or a fibration if and only if $X(S) \to
Y(S)$ is a weak equivalence or fibration of spaces for every finite
based set $S$. The cofibrations are determined by the left lifting
property.  Following Schwede's
terminology~\cite{Schwede_Gamma-spaces}, we call this the \emph{level
  Q-model structure} and write $(\GammaS)_{\lev}$ for this model
category. It is the first step towards the \emph{stable} Q-model
structure on $\GammaS$ appearing in
Theorem~\ref{thm:identification-of-group-compl-model-structure}.

Next we will build a corresponding ``level'' model structure on the
category $\GammaJS$. It is based on model structures on the categories
$\cS^{\HcJ(S)}$ we will treat first.

\begin{definition} An object $(\bld{s},\sigma)$ in $\HcJ(S)$ is
\emph{positive} if for every $i \in \ovl{S}$ the object
$\bld{s}_{i}=(\bld{m_1},\bld{m_2})$ of $\cJ$ satisfies $m_1 \geq 1$.
\end{definition}
A map $X \to Y$ in $\cS^{\HcJ(S)}$ is a \emph{$\HcJ(S)$-equivalence}
if the induced map of homotopy colimits $\hocolim_{\HcJ(S)}X \to
\hocolim_{\HcJ(S)}Y$ is a weak equivalence of spaces.  It is a
\emph{positive $\HcJ(S)$-fibration} if for every morphism $f\colon
(\bld{s},\sigma) \to (\bld{t},\tau)$ in $\HcJ(S)$ between positive
objects the induced square
\[\xymatrix@-1pc{
X(\bld{s},\sigma) \ar[r] \ar[d] & X(\bld{t},\tau) \ar[d]\\
Y(\bld{s},\sigma)\ar[r]& Y(\bld{t},\tau)
}
\]
is a homotopy cartesian square in which the vertical maps are Kan
fibrations. \emph{Positive $\HcJ(S)$-cofibrations} are the maps with
the left lifting property with respect to all maps which are both
positive $\HcJ(S)$-fibrations and $\HcJ(S)$-equivalences.

A map $X \to Y$ in $\cS^{\HcJ(S)}$ is an \emph{absolute
  $\HcJ(S)$-fibration} if it satisfies the condition of a positive
$\HcJ(S)$-fibration for all morphisms of $\HcJ(S)$ (without requiring
that their domains and codomains are positive). The
\emph{absolute $\HcJ(S)$-cofibrations} are the maps with the left
lifting property with respect to all maps which are both absolute
$\HcJ(S)$-fibrations and $\HcJ(S)$-equivalences.

\begin{lemma}\label{lem:HJS-model-structures} If $S$ is non-trivial, 
  the $\HcJ(S)$-equivalences, the positive $\HcJ(S)$-fibrations, and
  the positive $\HcJ(S)$-cofibrations define a cofibrantly generated
  proper simplicial \emph{positive $\HcJ(S)$-model structure}.  The
  same holds in the absolute case.
\end{lemma}
We write $\cS^{\HcJ(S)}_{\pos}$ and $\cS^{\HcJ(S)}_{\abs}$ for these
model categories.  It is immediate that the identity functors form a
Quillen equivalence
$\cS^{\HcJ(S)}_{\pos}\rightleftarrows\cS^{\HcJ(S)}_{\abs}$.
\begin{proof}
  We apply the general existence result for such ``homotopy colimit
  model structures'' on diagram categories~\cite[Proposition
  6.16]{Sagave-S_diagram}. Using this, it remains to show that
  $\HcJ(S)$ together with the discrete subcategory on the identity
  morphisms of positive objects (resp. all objects in the absolute
  case) is a \emph{well structured relative index category} in the
  sense of~\cite[Definition 5.2]{Sagave-S_diagram}. By~\cite[Corollary
  5.9]{Sagave-S_diagram}, this is the case for $\cJ$. If one defines
  the degree function of a product to be the sum of the degree
  functions of the factors, this structure is preserved under
  products. It is also preserved under equivalences of symmetric
  monoidal categories. Hence Lemma~\ref{lem:Shi-Shi-functorial} shows
  the claim for $\HcJ(S)$. Properness is shown in~\cite[\S
  11]{Sagave-S_diagram}. By~\cite[Proposition 6.19]{Sagave-S_diagram},
  the model structure is simplicial.
\end{proof}
\begin{remark}\label{rem:pos-abs-GammaJS} 
  We use a \emph{positive} model structure because of the connection
  to commutative $\cJ$-space monoids in
  Lemma~\ref{lem:Phi-Psi-level-Q-adj} below. The \emph{absolute}
  counterpart is needed in some proofs since it has the advantage that
  the $\alpha^*\colon \cS^{\HcJ(T)}_{\abs}\to\cS^{\HcJ(S)}_{\abs}$
  induced by $\alpha\colon S \to T$ preserves fibrations and weak
  equivalences. This is does not hold in the positive case because
  $\alpha^*(X)$ may evaluate $X$ at
  $\bld{s}_{\emptyset}=(\bld{0},\bld{0})$ in some components if
  $\alpha$ is not surjective.
\end{remark}

We now use these model categories for varying $S$ to construct model
structures on the category of $\GammaJ$-spaces. A map $f \colon X \to
Y$ in $\GammaJS$ is a \emph{level $\HcJ$-equivalence} if the map $X(S;
-) \to Y(S; -)$ is a $\HcJ(S)$-equivalence for every finite based set
$S$. It is a \emph{positive level $\HcJ$-fibration} if $X(S; -) \to
Y(S; -)$ is a positive $\HcJ(S)$-fibration for every finite based
set $S$. The \emph{positive level $\HcJ$-cofibrations} are the maps with
the left lifting property to all maps which are both level
  $\HcJ$-equivalences and positive level $\HcJ$-fibrations.

\begin{proposition}\label{prop:existence-HJ-level-model-str}
  The level $\HcJ$-equivalences, the positive level $\HcJ$-fibrations
  and the positive level $\HcJ$-cofibrations define a cofibrantly
  generated simplicial \emph{positive level model structure} on the
  category $\GammaJS$.
\end{proposition}
\begin{proof}
The adjunction of Lemma~\ref{lem:free-functor-FS} induces an adjunction
\[ F \colon \textstyle\prod_{S \neq 0^+} \cS^{\HcJ(S)}_*
\rightleftarrows \GammaJS \colon U \] in which $U$ is the evident
forgetful functor and $F((Z_S)_{S}) = \coprod F_S(Z_S)$. For every
$S$, the comma category $\cS^{\HcJ(S)}_* = \HcJ(S)(\bld{0}_S,-)
\downarrow \cS^{\HcJ(S)}$ inherits a cofibrantly generated positive
$\HcJ(S)$-model structure from $\cS^{\HcJ(S)}$. We apply the general
lifting criterion of a cofibrantly generated model structures along a
right adjoint provided by~\cite[Theorem
11.3.2]{Hirschhorn_model}. Since each of the $\cS^{\HcJ(S)}_*$ is
cofibrantly generated, so is their product~\cite[Proposition
11.1.10]{Hirschhorn_model}. The category $\GammaJS$ is locally
presentable, so all objects are small relative to the whole category.

It remains to show that for a set of generating acyclic cofibrations
$J$ of the product model structure, $U$ sends relative $FJ$-cell
complexes to weak equivalences. For this we use that the
$(\alpha_{!},\alpha^{*})$ are Quillen adjunctions for the absolute
model structures. The explicit description of $F_S$ in
Lemma~\ref{lem:free-functor-FS} shows that each component of a
relative cell complex built from maps which are in the image of a set
acyclic cofibration under $F$ is an absolute acyclic
$\HcJ(S)$-cofibration. Since the positive model structures have less
acyclic cofibrations and the same weak equivalences than the absolute
ones, this implies the claim.

The property of being simplicial is inherited from the $\HcJ(S)$-model
structures because it is enough to check this condition on the
cotensors.
\end{proof}
We write $(\GammaJS)_{\lev}$ for the model category resulting from
Proposition~\ref{prop:existence-HJ-level-model-str}. 
\begin{remark}
  We call this the \emph{level} model structure because it is related
  to $(\GammaS)_{\lev}$ by Proposition~\ref{prop:chain-level-Q-equiv}
  below. There are also level model structures on $\HcJ(S)$-spaces,
  but those will not be considered in this paper.
\end{remark}
The argument used in the proof of the last proposition also
implies
\begin{corollary}\label{cor:pos-lev_HcJ-cof-levelw-abs-cof}
  The positive level $\HcJ$-cofibrations are levelwise absolute
  $\HcJ(S)$-cofibrations.\qed
\end{corollary}
\begin{corollary}\label{cor:pos-lev-HcJ-proper}
  The positive level $\HcJ$-model structure is proper.
\end{corollary}
\begin{proof}
  Right properness lifts from the positive $\HcJ(S)$-model
  structure. With the previous corollary, left properness of the
  absolute $\HcJ(S)$-model structure implies left properness of the
  positive level $\HcJ$-model structure. 
\end{proof}

\begin{lemma}\label{lem:Phi-Psi-level-Q-adj}
  The adjunction $\Phi\colon(\GammaJS)_{\lev} \rightleftarrows
  \cC\cS^{\cJ}\colon\Psi$
  of Corollary~\ref{cor:Phi-Psi-adj} is a Quillen adjunction with
  respect to the positive $\cJ$-model structure on $\cC\cS^{\cJ}$.
\end{lemma}
\begin{proof}
The identification of the right adjoint in Corollary~\ref{cor:identification-of-Psi} makes
it easy to check that it preserves weak equivalences and fibrations. 
\end{proof}

\subsection{\texorpdfstring{$\GammaJ$}{Gamma-J}-spaces and augmented
  \texorpdfstring{$\Gamma$}{Gamma}-spaces} We will now relate
$\GammaJ$-spaces to ordinary $\Gamma$-spaces by applying the
comparison between $\cK$-spaces and spaces over $B\cK$
of~\cite[Theorem 13.2]{Sagave-S_diagram} levelwise.

As in Definition~\ref{def:K-space}, let $\cS^{\cK}$ be the category of
$\cK$-spaces for a small category $\cK$. Viewing the comma category
$(\cK\downarrow \bld{k})$ as a functor in $\bld{k}$, we obtain a
$\cK$-space $E\cK = B(\cK\downarrow -)$. Composition with $E\cK \to *$ and
the product with $E\cK$ provide the first adjunction in  
\begin{equation}\label{eq:diagram-spaces-graded-spaces}
\cS^{\cK} \leftrightarrows \cS^{\cK}/E\cK \rightleftarrows \cS/B\cK. 
\end{equation}
The right adjoint in the second adjunction is obtained by viewing a
map of spaces $Y \to B\cK$ as a map of constant $\cK$-spaces and
forming the pullback $E\cK\times_{B\cK}Y$ along the map $E\cK \to
B\cK$ induced by the projection $(\cK\downarrow \bld{k}) \to \cK$. Its
left adjoint is obtained by applying $\colim_{\cK}$ and composing with 
$\colim_{\cK}E\cK \xrightarrow{\iso} B\cK$.

\begin{lemma}\label{lem:HcJS-spaces-vs-graded-sp}
  Let $S$ be a non-trivial finite based set. Then for $\cK = \HcJ(S)$,
  the diagram~\eqref{eq:diagram-spaces-graded-spaces} is a chain of
  Quillen equivalences with respect to the positive $\HcJ(S)$-model
  structure on $\cS^{\HcJ(S)}$, the positive overcategory model
  structure on $\cS^{\HcJ(S)}/E\HcJ(S)$, and the overcategory model
  structure on $\cS/B\HcJ(S)$. The same holds in the absolute case.
\end{lemma}
\begin{proof}
  We used in the proof of Lemma~\ref{lem:HJS-model-structures} that
  $\HcJ(S)$ is a \emph{well structured index category} in the
  sense of~\cite[Definition 5.5]{Sagave-S_diagram}.
  Hence~\cite[Theorem 13.2]{Sagave-S_diagram} applies and provides the
  desired Quillen equivalences for the \emph{absolute} $\HcJ(S)$-model
  structures on $\HcJ(S)$-spaces. Since the \emph{positive} model
  structures have the same weak equivalences and less cofibrations
  than the absolute ones, this implies the claim.
\end{proof}

The $\cK$-space $E\cK$ appearing
in~\eqref{eq:diagram-spaces-graded-spaces} is functorial in $\cK$: A
functor $F\colon \cK \to \cL$ induces a natural transformation $E\cK
\to F^*(E\cL)$ of $\cK$-spaces. Inspecting the definition of the
Grothendieck construction, one verifies that these natural
transformations make the collection of $\HcJ(S)$-spaces $E\HcJ(S)$
into a $\GammaJ$-space~$E\HcJ$.

\begin{proposition}\label{prop:chain-level-Q-equiv}
There is a chain of Quillen equivalences
\begin{equation}\label{eq:diagram-spaces-graded-spaces-lifted}
  (\GammaJS)_{\lev} \leftrightarrows \left(\GammaJS\right)_{\lev} / E\HcJ \rightleftarrows \left(\GammaS\right)_{\lev} / \bof{\cJ}
\end{equation}
induced by the chain of adjunctions~\eqref{eq:diagram-spaces-graded-spaces}.
\end{proposition}
\begin{proof}
  The first adjunction is induced by the level $\HcJ$-equivalence
  $E\HcJ \to *$.  It is a Quillen equivalence since
  $(\GammaJS)_{\lev}$ is right proper.

  Using the description of the category $\GammaJS$ in
  Lemma~\ref{lem:GammaJS-equivalent} one can check that the
  adjunctions $\cS^{\HcJ(S)}/E\HcJ(S) \rightleftarrows \cS/B\HcJ(S)$
  defined in~\eqref{eq:diagram-spaces-graded-spaces} and the maps
  $E\HcJ(S)\to\alpha^*(E\HcJ(T))$ do indeed induce the second
  adjunction in~\eqref{eq:diagram-spaces-graded-spaces-lifted}.  The
  positive instance of the Quillen equivalence
  $\cS^{\HcJ(S)}/E\HcJ(S)\rightleftarrows \cS/B\HcJ(S)$ of
  Lemma~\ref{lem:HcJS-spaces-vs-graded-sp} and the definition of the
  model structures imply that the right adjoint is a right Quillen
  functor. Corollary~\ref{cor:pos-lev_HcJ-cof-levelw-abs-cof} and the
  absolute case of Lemma~\ref{lem:HcJS-spaces-vs-graded-sp} show that
  it is a Quillen equivalence.
\end{proof}

We write $\cK$ for $\HcJ(S)$ for the rest of this section. The
composed derived functor of the Quillen
equivalences~\eqref{eq:diagram-spaces-graded-spaces} sends a
$\cK$-space $Z$ to $\colim_{\cK}\left((Z^{\fib} \times
  E\cK)^{\cof})\right)$. (Here \emph{composed derived functor} has the
same meaning as in
Theorem~\ref{thm:identification-of-group-compl-model-structure}.)

The maps $(Z^{\fib} \times E\cK)^{\cof} \to Z^{\fib} \times E\cK \to
Z^{\fib} \ot Z$ and the natural map from the homotopy colimit to the
colimit induce weak equivalences
\begin{equation}\label{eq:hocolim-colim-zig-zag}
Z_{h\cK}\to(Z^{\fib})_{h\cK} \ot \left((Z^{\fib} \times E\cK)^{\cof})\right)_{h\cK} \to
\colim_{\cK}\left((Z^{\fib} \times E\cK)^{\cof})\right)
\end{equation}
by~\cite[Lemma 6.22]{Sagave-S_diagram} and the right properness of 
the $\HcJ(S)$-model structure. 
\begin{corollary}\label{cor:hocolim-colim-zig-zag}
  After forgetting the augmentation to $B\HcJ(S)$, the value of the
  composed derived functor of the chain of Quillen equivalences of
  Lemma~\ref{lem:HcJS-spaces-vs-graded-sp} at an $\HcJ(S)$-space $Z$
  is weakly equivalent to the homotopy colimit $Z_{h\HcJ(S)}$.\qed
\end{corollary}
To simplify the notation we now assume that $Z$ is positive
$\HcJ(S)$-fibrant and compare $Z_{h\cK}$ and $\colim_{\cK}\left((Z
  \times E\cK)^{\cof})\right)$ as spaces over $B\cK$. This is
more subtle since $Z \to *$ induces the augmentation
$Z_{h\cK} \to (*)_{h\cK}\iso B\cK$ of the former, while the
augmentation of latter is induced by the maps $(Z \times
E\cK)^{\cof} \to E\cK$ and $\colim_{\cK}E\cK
\xrightarrow{\iso} B\cK$.  These maps fit into a commutative diagram
\begin{equation}\label{eq:hocolim-colim-over-BK}
  \xymatrix@-1pc{&Z_{h\cK}    \ar[d] & (Z \times E\cK)_{h\cK}\ar[l]_-{\sim} \ar[d]
    & \left((Z \times E\cK)^{\cof})\right)_{h\cK} \ar[l]_-{\sim}
    \ar[r]^-{\sim} \ar[d]
    & \colim_{\cK}\left((Z \times E\cK)^{\cof})\right) \ar[d] \\
    B\cK \ar@{}[r]|-{\iso}& (*)_{h\cK} & (E\cK)_{h\cK} \ar[l]_-{\sim}
    \ar@{=}[r]& (E\cK)_{h\cK} \ar[r]^-{\sim} & B\cK }
\end{equation}
in which the two maps $(E\cK)_{h\cK} \to B\cK$ in the bottom line
differ. So the zig-zag of weak
equivalences~\eqref{eq:hocolim-colim-zig-zag} does not lie over
$B\cK$. However, using the notation $(*)_{h\cK}$ to distinguish the
two copies of $B\cK$, we have
\begin{lemma}\label{lem:hocolim-colim-over-BK}
  Composition with and base change along the bottom maps
  in~\eqref{eq:hocolim-colim-over-BK} induce a chain of Quillen
  equivalences
  \[ \cS/(*)_{h\HcJ(S)} \leftrightarrows \cS/(E\HcJ(S))_{h\HcJ(S)}
  \rightleftarrows \cS/B\HcJ(S). \] For a $\HcJ(S)$-space $Z$, there
  is a chain of $\HcJ(S)$-equivalences relating the value of
  $Z_{h\HcJ(S)} \to (*)_{h\HcJ(S)}$ under the composed derived functor
  of this chain of Quillen equivalences with
  $\colim_{\cK}\left((Z^{\fib} \times
    E\cK)^{\cof})\right) \to B\HcJ(S)$.
\end{lemma}
\begin{proof}
  Right properness of $\cS$ implies that composition with and pullback
  along a weak equivalence induces a Quillen equivalence of comma
  categories.  By~\cite[Corollary 11.4]{Sagave-S_diagram}, the left
  hand square in~\eqref{eq:hocolim-colim-over-BK} is homotopy
  cartesian. Hence~\eqref{eq:hocolim-colim-over-BK} computes the
  composed derived functor and provides the desired chain of
  equivalences.
\end{proof}

\section{Stable model structures on
  \texorpdfstring{$\GammaJ$}{Gamma-J}-spaces}\label{sec:stable-model-str}
In this section we build \emph{pre-stable} and \emph{stable} model
structures on the category of $\GammaJ$-spaces and use them to prove
Theorem~\ref{thm:group-completion-model-str}
and Theorem~\ref{thm:identification-of-group-compl-model-structure}.

\subsection{Left Bousfield localizations} Below we will frequently
employ \emph{left Bousfield localizations} and therefore recall some
terminology about these. To simplify this, we only consider
localizations of simplicial model categories.  We refer to
Hirschhorn's book~\cite{Hirschhorn_model} as an extensive reference
about left Bousfield localizations.

Let $\cM$ be left proper simplicial model category and $S$ be is a set
of maps between cofibrant objects in $\cM$. An object $W$ of $\cM$ is
called \emph{$S$-local} if it is fibrant in $\cM$ and if every
$f\colon X \to Y$ in $S$ induces a weak equivalence $f^*\colon
\Map(Y,W) \to \Map(X,W)$ of simplicial mapping spaces. A map of
cofibrant objects $g\colon X \to Y$ is a \emph{$S$-local equivalence}
if the induced map $g^*\colon \Map(Y,W) \to \Map(X,W)$ of simplicial
mapping spaces is a weak equivalence for every $S$-local object $W$. A
general map $g$ is a $S$-local equivalence if the induced map of
cofibrant replacements is.

If it exists, the \emph{left Bousfield localization} of $\cM$ with
respect to $S$ is a new model category structure on $\cM$ with the
same cofibrations as before and the $S$-local equivalences as weak
equivalences. The fibrations are determined by a lifting
property. Left properness of $\cM$ ensures that the fibrant objects of
the localization are the $S$-local objects~\cite[Proposition
3.4.1]{Hirschhorn_model}. Left Bousfield localizations exist if $\cM$
is sufficiently well behaved, for example left proper and
\emph{cellular}~\cite[Theorem 4.1.1]{Hirschhorn_model} or left proper
and \emph{combinatorial}~\cite[Theorem 4.7]{Barwick_left-right}.

\subsection{Pre-stable model structures}
Lemma~\ref{lem:Phi-Psi-level-Q-adj} and
Proposition~\ref{prop:chain-level-Q-equiv} provide a chain of one
Quillen adjunction and two Quillen equivalences
\begin{equation}\label{eq:four-Q-adj-level}
  \cC\cS^{\cJ} \leftrightarrows (\GammaJS)_{\lev} \leftrightarrows 
  \left(\GammaJS\right)_{\lev} / E\HcJ \rightleftarrows \left(\GammaS\right)_{\lev} / \bof{\cJ} 
\end{equation}
with respect to the positive $\cJ$-model on $\cC\cS^{\cJ}$ and the
(positive) level model structures on the remaining categories.  We
will now localize the three model categories on the right hand site so
that all adjunctions in~\eqref{eq:four-Q-adj-level} become Quillen
equivalences. The idea is to build certain \emph{pre-stable} model
structures that lie between the level model structures of the last
paragraph and the stable model structures that occur in the
identification of the group completion model structure in
Theorem~\ref{thm:identification-of-group-compl-model-structure}. 

We begin with ordinary $\Gamma$-spaces. Let $p_S\colon S\wdg T \to S$
and $p_T\colon S\wdg T \to T$ be the projections associated with a
pair of objects in $\Gamma^{\op}$.  We consider the set of maps
\[ P'=\{p_S^* \wdg p_T^* \colon \Gamma^{\op}(S,-) \wdg
\Gamma^{\op}(T,-) \to \Gamma^{\op}(S\wdg T, -) \; | \; S, T \in
\obj(\Gamma^{\op})\}\] 
in $\GammaS$. 
\begin{lemma}\label{lem:GammaSlev-existence}
  The left Bousfield localization of $(\GammaS)_{\lev}$ with respect to
  $P'$ exists. It is a cofibrantly generated left proper model category.
\end{lemma}
\begin{proof}
  The category $\GammaS$ is locally presentable, and
  $(\GammaS)_{\lev}$ is cofibrantly generated and left proper. Hence
  Smith's existence theorem for left Bousfield localizations of
  combinatorial model categories~\cite[Theorem
  4.7]{Barwick_left-right} applies.
\end{proof}
We write $\left(\GammaS\right)_{\pre}$ for the resulting
\emph{pre-stable} model structure on $\GammaS$ and call its weak
equivalences \emph{pre-stable equivalences}. The above
characterization of fibrant objects in the localization implies that
its fibrant objects are the levelwise fibrant $\Gamma$-spaces which
are special (as defined in Section~\ref{subsec:GammaS}).
\begin{remark}
  The pre-stable model structure is related to Segal's original
  definition of $\Gamma$-spaces~\cite{Segal_categories}. In our
  terminology, Segal only considers special $\Gamma$-spaces, and the
  homotopy category of the model category
  $\left(\GammaS\right)_{\pre}$ is a model for the homotopy category
  of special $\Gamma$-spaces that Segal uses. The homotopy category of
  special $\Gamma$-spaces is also considered by
  Mandell~\cite{Mandell_inverse-K-theory} who proves that it is
  equivalent to the homotopy categories of $E_{\infty}$ spaces and of
  permutative categories.
\end{remark}

To obtain pre-stable model structures on $\GammaJ$-spaces, we need to
choose a set of maps similar to $P'$.  If $(\bld{u},\nu)$ is an object
of $\HcJ(S\wdg T)$, the projections $p_S$ and $p_T$ induce morphisms
$p_S = (p_S;\id)\colon (S\wdg T;\bld{u},\nu) \to
(S;(p_S)_*(\bld{u},\nu))$ and $p_T = (p_T;\id)$ in $\GammaJ$. These
maps in turn induce morphisms
\begin{equation}\label{eq:map-for-pre-st-loc}
\xymatrix@-1pc{
\GammaJ\left((S;(p_S)_*(\bld{u},\nu)),-\right) \textstyle\coprod 
\GammaJ\left((T;(p_T)_*(\bld{u},\nu)),-\right) \ar[d]^{p_S^*\wdg p_T^*} \\
\GammaJ\left((S\wdg T;\bld{u},\nu),-\right)}
\end{equation}
of the free $\GammaJ$-spaces represented by the objects involved. 
\begin{definition}\label{def:pre-loc-maps}
  Let $P$ be the set of maps in $\GammaJS$ of the
  form~\eqref{eq:map-for-pre-st-loc} where $S$ and $T$ are non-trivial
  objects of $\Gamma^{\op}$ and $(\bld{u},\nu)$ is a positive object
  of $\HcJ(S\wdg T)$. Let $P_E$ be the set of all maps in $
  \left(\GammaJS\right) / E\HcJ$ whose projection to $\GammaJS$ is a
  map in $P$, and let $P_B$ be the image of $P_E$ under the left
  adjoint $\left(\GammaJS\right) / E\HcJ \to \left(\GammaS\right) /
  \bof{\cJ}$.
\end{definition}
\begin{lemma}\label{lem:P-is-image}
  The image of $P_E$ under $\GammaJS/E\HcJ \to \GammaJS$ is $P$.
\end{lemma}
\begin{proof} 
  It is enough to show that there exists a map $\GammaJ\left((S\wdg
    T;\bld{u},\nu),-\right) \to E\HcJ$ in $\GammaJS$. This holds
  because the space $E\HcJ(S\wdg T;\bld{u},\nu) \iso
  B(\HcJ(S\wdg T)\downarrow (\bld{u},\nu))$ is non-empty.
\end{proof}
\begin{lemma}\label{lem:pre-stable-model-str}
  The left Bousfield localizations of $(\GammaJS)_{\lev}$ with respect
  to $P$, of $\left(\GammaJS\right)_{\lev} / E\HcJ$ with respect to
  $P_E$, and of $\left(\GammaS\right)_{\lev} / \bof{\cJ}$ with respect
  to $P_B$ exist. They are cofibrantly generated and left proper.
\end{lemma}
\begin{proof}
  Since the categories are locally presentable and the level model
  structures are cofibrantly generated and left proper, Smith's
  existence theorem for left Bousfield localizations~\cite[Theorem
  4.7]{Barwick_left-right} applies. 
\end{proof}
We refer to these model structures as the \emph{pre-stable} model
structures, write $(\GammaJS)_{\pre}, \left(\left(\GammaJS\right) /
  E\HcJ\right)_{\pre}$, and $\left(\left(\GammaS\right) /
  \bof{\cJ}\right)_{\pre}$ for these model categories and call the
weak equivalences \emph{pre-stable equivalences}.

It is important to note (and non-trivial to prove) that two possibly
distinct model structures on $\GammaS/\bof{\cJ}$ arising from these
constructions coincide:
\begin{lemma}\label{lem:pre-stable-vs-overcategory}
  The overcategory model structure
  $\left(\GammaS\right)_{\pre}/\bof{\cJ}$ and the pre-stable model
  structure $\left(\GammaS/\bof{\cJ}\right)_{\pre}$ of the last lemma
  coincide.
\end{lemma}
\begin{proof}
  Let $P'_B$ be the set of all maps in $\GammaS/\bof{\cJ}$ whose
  projection to $\GammaS$ is in $P'$. Since the levelwise colimit of a
  free $\GammaJ$-space is a free $\Gamma$-space, the left adjoint
  $\GammaJS/ E\HcJ \to\GammaS/\bof{\cJ}$ sends maps in $P_E$ to maps
  in $P'_B$.  We claim that up to isomorphism, every map in $P'_B$ is
  in the image.  For this we have to show that every possible
  augmentation $\Gamma^{\op}(S,-)\to\bof{\cJ}$ of the codomain of a
  map in $P'$ is in the image. Such maps correspond to objects
  $(\bld{s},\sigma)$ of $\HcJ(S)$, and the map
  $\GammaJ(S;\bld{s},\sigma),-) \to E\HcJ$ specified by the point in
  $B(\HcJ(S)\downarrow (\bld{s},\sigma))$ given by the identity on
  $(\bld{s},\sigma)$ is the desired preimage under the left adjoint.
  This proves the claim about $P'_B$.

  It remains to prove that the localization of $\GammaS/\bof{\cJ}$
  with respect to $P'_B$ is the overcategory model structure
  $\left(\GammaS\right)_{\pre}/\bof{\cJ}$. Since all maps in $P'_B$
  are weak equivalences in $\left(\GammaS\right)_{\pre}$ when
  forgetting the augmentation, the identity functor is a left Quillen
  functor $\left(\GammaS/\bof{\cJ}\right)_{\pre} \to
  \left(\GammaS\right)_{\pre}/\bof{\cJ}$. To see that the model
  structures coincide, it is therefore enough to show that the fibrant
  objects in $\left(\GammaS/\bof{\cJ}\right)_{\pre}$ are fibrant in
  $\left(\GammaS\right)_{\pre}/\bof{\cJ}$. An object $f\colon X \to
  \bof{\cJ}$ is fibrant in $\left(\GammaS/\bof{\cJ}\right)_{\pre}$ if
  $f$ is a level fibration in $\GammaS$ and every map $g\colon U \to
  V$ in $P'_B$ induces a weak equivalence $\Map_{\bof{\cJ}}(g,X)$. The
  space $\Map_{\bof{\cJ}}(U,X)$ of maps from $U\to \bof{\cJ}$ to $f$
  in $\GammaS/\bof{\cJ}$ can be defined as the fiber of $f_*\colon
  \Map_{\GammaS}(U,X) \to \Map_{\GammaS}(U,\bof{\cJ})$ over the
  augmentation $U \to \bof{\cJ}$ of $U$. Since $f$ is a level
  fibration, $f_*$ is a fibration of spaces. Hence every $g\colon U
  \to V$ in $P'_B$ induces a map of homotopy fiber sequences
  \[\xymatrix@-1pc{
    \Map_{\bof{\cJ}}(V,X) \ar[r] \ar[d] & \Map_{\GammaS}(V,X) \ar[r]
    \ar[d]
    & \Map_{\GammaS}(V,\bof{\cJ}) \ar[d] \\
    \Map_{\bof{\cJ}}(U,X) \ar[r] & \Map_{\GammaS}(U,X) \ar[r] &
    \Map_{\GammaS}(U,\bof{\cJ}). }\] By assumption, the first and
  third vertical maps are weak equivalences. Since the maps $g$ in
  $P'_B$ that project to a given map $U \to V$ in $P'$ vary over all
  possible maps over $\bof{\cJ}$, we obtain the previous diagram of
  homotopy fiber sequences for all possible basepoints in
  $\Map_{\GammaS}(V,\bof{\cJ})$. It follows that the middle vertical
  map is a weak equivalence for all $U\to V$ in $P'$. Let $f'\colon X'
  \to \bof{\cJ}'$ be a level fibration of level fibrant objects that
  is level equivalent to $f$. Since the $U$ is a coproduct of free
  $\Gamma$-spaces, the mapping space $\Map_{\GammaS}(U,-)$ maps level
  equivalences between all objects to weak equivalences, and it
  follows that $X'$ and $\bof{\cJ}'$ are fibrant in
  $\left(\GammaS\right)_{\pre}$. So $f'$ is a fibration in
  $\left(\GammaS\right)_{\pre}$ by~\cite[Proposition
  3.4.7]{Hirschhorn_model}. Forming the base change along
  $\bof{\cJ}\to\bof{\cJ}'$, ~\cite[Proposition
  3.4.6]{Hirschhorn_model} shows that $f$ is a fibration in
  $\left(\GammaS\right)_{\pre}$.
\end{proof}

\begin{corollary}\label{cor:chain-pre-stable-equiv}
There is a chain of Quillen equivalences 
\begin{equation}\label{eq:chain-pre-stable-equiv}
  (\GammaJS)_{\pre} \leftrightarrows 
  \left(\GammaJS\right)_{\pre} / E\HcJ \rightleftarrows \left(\GammaS\right)_{\pre} / \bof{\cJ} 
\end{equation}
with respect to the pre-stable model structures.\qed
\end{corollary}
\begin{proof}
  All maps in the sets $P$, $P_E$, and $P_B$ have cofibrant
  domains and codomains in the respective categories. Hence this is a
  consequence of the last lemma and the fact that Quillen equivalences
  are preserved under left Bousfield localizations~\cite[Theorem
  3.3.20(1)(b)]{Hirschhorn_model} applied to the Quillen equivalences
  in~\eqref{eq:diagram-spaces-graded-spaces-lifted}.
\end{proof}
The corollary enables us to give a more explicit characterization of
the weak equivalences and the fibrant objects in $(\GammaJS)_{\pre}$
in terms of ordinary $\Gamma$-spaces and the functor $\gamma'\colon
\GammaJS \to \GammaS$ introduced in connection with
Corollary~\ref{cor:gamma-vs-gamma-prime}:
\begin{lemma}\label{lem:gamma-prime-detects-we}
  A map $f\colon X \to Y$ in $\GammaJS$ is a pre-stable equivalence if
  and only if $\gamma'(f)$ is a pre-stable equivalence in $\GammaS$. A
  level fibrant object $X$ is fibrant in $(\GammaJS)_{\pre}$ if and
  only if the $\Gamma$-space $\gamma'(X)$ is special.
\end{lemma} 
\begin{proof}
  The strategy of the proof is to compare the zig-zag of Quillen
  equivalences of level model
  structures~\eqref{eq:diagram-spaces-graded-spaces-lifted} with the
  zig-zag of pre-stable model
  structures~\eqref{eq:chain-pre-stable-equiv}.

  A map $f$ in $\GammaJS$ is a pre-stable equivalence if and only if
  it represents an isomorphism in $\Ho((\GammaJS)_{\pre})$. Let $g$ be
  the image of $f$ under the composed derived functor of the zig-zag
  of level model structures. Then $g$ represents the image of the
  homotopy class of $f$ under the equivalence $\Ho((\GammaJS)_{\lev})
  \to \Ho((\GammaS)_{\lev}/\bof{\cJ})$ induced
  by~\eqref{eq:diagram-spaces-graded-spaces-lifted}. Using the
  equivalence $\Ho((\GammaJS)_{\pre}) \to
  \Ho((\GammaS)_{\pre}/\bof{\cJ})$ induced
  by~\eqref{eq:chain-pre-stable-equiv}, we see that $f$ represents an
  isomorphism in $\Ho((\GammaJS)_{\pre})$ if and only if $g$
  represents an isomorphism in
  $\Ho((\GammaJS)_{\pre}/\bof{\cJ})$. Since the chain of weak
  equivalences~\eqref{eq:hocolim-colim-zig-zag} induces a chain of
  level equivalences of $\Gamma$-spaces between $\gamma'(f)$ and the
  underlying map of $\Gamma$-spaces associated with $g$, the claim
  follows.

  An object $X$ is fibrant in $(\GammaJS)_{\pre}$ if and only if it is
  $P$-local in the level-model structure. By~\cite[Proposition
  3.1.12]{Hirschhorn_model} and Lemma~\ref{lem:P-is-image}, $X$ is
  $P$-local if and only if the image $X\times E\HcJ$ of $X$ under the
  right adjoint to $\left((\GammaJS) / E\HcJ\right)_{\lev}$ is
  $P_E$-local. Combining~\cite[Proposition 3.1.12]{Hirschhorn_model}
  with~\cite[Proposition 3.2.2]{Hirschhorn_model}, it follows that an
  object of $\left((\GammaJS) / E\HcJ\right)_{\lev}$ is $P_E$-local if
  and only if the left adjoint to $\left((\GammaS) /
    \bof{\cJ}\right)_{\lev}$ sends a cofibrant replacement of it to an
  object that is $P_B$-local after level fibrant
  replacement. Lemma~\ref{lem:pre-stable-vs-overcategory},
  ~\cite[Proposition 3.3.16]{Hirschhorn_model} and the fact that
  $\bof{\cJ}$ is (very) special show that a fibration over $\bof{\cJ}$
  in $(\GammaS)_{\pre}$ is a fibrant object in $\left((\GammaS) /
    \bof{\cJ}\right)_{\pre}$ if an only if its source is special.
  Hence $X$ if fibrant in $(\GammaJS)_{\pre}$ if and only if it is
  level fibrant and the underlying $\Gamma$-space of the image of $X$
  under the composed derived functor of the zig-zag of
  Quillen-adjunctions of the level model structures is
  special. Comparing with $\gamma'$ as above proves the claim.
\end{proof}
The following result about the relation to commutative $\cJ$-space
monoids is the most difficult step towards
Theorem~\ref{thm:group-completion-model-str} and will be proved at the
end of this section:
\begin{proposition}\label{prop:J-model-pre-stable-equiv}
  The adjunction $\Phi\colon\GammaJS \rightleftarrows
  \cC\cS^{\cJ}\colon\Psi$ is a Quillen equivalence with respect to the
  positive $\cJ$-model structure on $\cC\cS^{\cJ}$ and the positive
  pre-stable model structure on $\GammaJS$.
\end{proposition}
Together with Corollary~\ref{cor:chain-pre-stable-equiv}, the proposition implies
\begin{corollary}\label{cor:chain-pre-stable-equiv-and-CSJ}
There is a chain of Quillen equivalences 
\begin{equation}\label{eq:chain-pre-stable-equiv-and-CSJ}
  \cC\cS^{\cJ} \leftrightarrows (\GammaJS)_{\pre} \leftrightarrows 
  \left(\GammaJS/ E\HcJ\right)_{\pre}  \rightleftarrows \left(\GammaS\right)_{\pre} / \bof{\cJ} 
\end{equation}
with respect to the positive $\cJ$-model structure on $\cC\cS^{\cJ}$.\qed
\end{corollary}

\subsection{Stable model structures on
  \texorpdfstring{$\GammaJ$}{Gamma-J}-spaces}\label{subsec:stable-model-Gamma}
To get to the stable model structures on $\GammaJ$-spaces and the group completion
model structure on $\cC\cS^{\cJ}$ we are after, we will localize the model categories
in Corollary~\ref{cor:chain-pre-stable-equiv} one more time. 

For this we consider the maps of finite based sets $d\colon 2^+ \to
1^+$ and $ p_1\colon 2^+ \to 1^+$ defined by $d(1) = 1 = d(2)$ and
$p_1(1) = 1, p_1(2) = 0$. The same arguments as for
Lemma~\ref{lem:GammaSlev-existence} imply
\begin{lemma}
  The left Bousfield localization of $(\GammaS)_{\pre}$ with
  respect to the map $d^*\wdg p_1^*\colon \Gamma^{\op}(1^+,-)\wdg
  \Gamma^{\op}(1^+,-) \to \Gamma^{\op}(2^+,-)$ exists. \qed 
\end{lemma}
We call this the stable Q-model structure and write $(\GammaS)_{\sta}$
for this model category. It is immediate from the construction that
this coincides with the stable Q-model structure on $\Gamma$-spaces
considered in \cite[Theorem 1.5]{Schwede_Gamma-spaces}. This means
that a map in $(\GammaS)_{\sta}$ is a weak equivalence if the
associated map of spectra is a stable equivalence and that an object
is fibrant if and only if it level fibrant and very special. Moreover,
the Quillen adjunction $\GammaS \rightleftarrows \SpN$ exhibits the
homotopy category of $(\GammaS)_{\sta}$ as the homotopy category of
connective spectra. Note that however, $(\GammaS)_{\sta}$ is
\emph{not} a stable model category in the sense that the suspension
and loop functor induce inverse equivalences on homotopy categories.

We need again more elaborate sets of maps to localize the various
categories of $\GammaJ$-spaces: As in~\eqref{eq:map-for-pre-st-loc},
every object $(\bld{u},\nu)$ of $\HcJ(2^+)$ gives rise to a map
\begin{equation}\label{eq:map-for-st-loc}
\xymatrix@-1pc{
\GammaJ\left((1^+;(d)_*(\bld{u},\nu)),-\right) \textstyle\coprod
\GammaJ\left((1^+;(p_1)_*(\bld{u},\nu)),-\right) \ar[d]^{d^*\wdg p_1^*} \\
\GammaJ\left((2^+;\bld{u},\nu),-\right)}
\end{equation}
of free $\GammaJ$-spaces. 
\begin{definition}\label{def:st-loc-maps}
  Let $Q$ be the set of maps in $\GammaJS$ of the
  form~\eqref{eq:map-for-st-loc} where $(\bld{u},\nu)$ is a positive
  object of $\HcJ(2^+)$. Let $Q_E$ and $Q_B$ be defined as the
  corresponding sets in Definition~\ref{def:pre-loc-maps}.
\end{definition}
Applying the functor $\Phi\colon \GammaJS \to \cC\cS^{\cJ}$ to the set
$Q$ provides a set of maps in~$\cC\cS^{\cJ}$. Up to isomorphism, a map
in $\Phi(Q)$ is of the form
\begin{equation}\label{eq:shear-maps-for-CSJ}
  \xymatrix@-1pc{
    \left( \mC F^{\cJ}_{(\bld{m_1}\concat\bld{n_1},\bld{m_2}\concat\bld{n_2})}(*)\right) \boxtimes \left(\mC F^{\cJ}_{(\bld{n_1},\bld{n_2})}(*)\right) \ar[d]\\   \left( \mC F^{\cJ}_{(\bld{m_1},\bld{m_2})}(*)\right) \boxtimes  \left(\mC F^{\cJ}_{(\bld{n_1},\bld{n_2})}(*)\right)}
\end{equation}
where $(\bld{m_1},\bld{m_2})$ and $(\bld{n_1},\bld{n_2})$ run through
the positive objects in $\cJ$. 

\begin{lemma}\label{lem:grp-loc-existence}
The left Bousfield localization of the positive $\cJ$-model structure on $\cC\cS^{\cJ}$
with respect to $\Phi(Q)$ exists. 
\end{lemma}
\begin{proof}
  The positive $\cJ$-model structure on $\cC\cS^{\cJ}$ is cofibrantly
  generated and left proper~\cite[Proposition
  4.10]{Sagave-S_diagram}. Since the forgetful functor $\cC\cS^{\cJ}
  \to \cS^{\cJ}$ preserves filtered colimits, it follows
  from~\cite[5.2.2b and Theorem 5.5.9]{Borceux_handbook-II} that
  $\cC\cS^{\cJ}$ is locally presentable.  Hence the existence
  theorem~\cite[Theorem 4.7]{Barwick_left-right} applies also here.
\end{proof}
The resulting model structure $\cC\cS^{\cJ}_{\grp}$ will be the group
completion model structure of
Theorem~\ref{thm:group-completion-model-str}. 

Counterparts of Lemma~\ref{lem:P-is-image},
Lemma~\ref{lem:pre-stable-model-str} and
Lemma~\ref{lem:pre-stable-vs-overcategory} for the sets $P\cup Q$,
$P_E\cup Q_E$, and $P_B\cup Q_B$ can be proved by the same arguments
we used in these lemmas for $P$, $P_E$, and $P_B$.  This provides
stable model structures on $\GammaJS$ and $\GammaJS/E\HcJ$, and the
same arguments as for Corollary~\ref{cor:chain-pre-stable-equiv} imply
\begin{corollary}\label{cor:chain-sta-stable-equiv} There is a chain of
  Quillen equivalences
\begin{equation}\label{eq:chain-sta-stable-equiv}
  \cC\cS^{\cJ}_{\grp} 
  \leftrightarrows \left(\GammaJS/ E\HcJ\right)_{\sta} 
  \rightleftarrows \left(\GammaS\right)_{\sta} / \bof{\cJ}.
  \end{equation}
with respect to the group completion model structure on $\cC\cS^{\cJ}$.\qed
\end{corollary}
\begin{proof}[Proof of Theorem~\ref{thm:group-completion-model-str}]
  Lemma~\ref{lem:grp-loc-existence} provides the existence of the
  model structure. It remains to identify the weak equivalences, the
  fibrant objects and the fibrant replacement. The characterization of
  the weak equivalences and the fibrant objects in terms of $\gamma$
  is completely analogous to Lemma~\ref{lem:gamma-prime-detects-we},
  this time comparing the zig-zag of Quillen equivalences of
  pre-stable model
  structures~\eqref{eq:chain-pre-stable-equiv-and-CSJ} with the
  zig-zag of stable model
  structures~\eqref{eq:chain-sta-stable-equiv}.  This uses that the
  stable model structures can be obtained from the pre-stable ones by
  localizing with respect to $Q_E$ and~$Q$.

  The identification of the weak equivalences and fibrant objects in
  $\cC\cS^{\cJ}_{\grp}$ implies that the fibrant replacement is a
  group completion in the sense of
  Definition~\ref{def:group-completion}.
\end{proof}
\begin{remark}
  A more direct way to see that fibrant objects in
  $\cC\cS^{\cJ}_{\grp}$ are grouplike is that the maps
  in~\eqref{eq:map-for-st-loc} corepresent shear maps $(x,y)\mapsto
  (x,x+y)$ when applying $\pi_0(\Map(-,A))$. 
\end{remark}

\begin{proof}[Proof of Theorem~\ref{thm:identification-of-group-compl-model-structure}]
  Corollary~\ref{cor:chain-sta-stable-equiv} already gives one
  possible choice for a zig-zag of Quillen equivalences relating
  $\cC\cS^{\cJ}_{\grp}$ and $\left(\GammaS\right)_{\sta} /
  \bof{\cJ}$. We used in the proof of
  Theorem~\ref{thm:group-completion-model-str} that its composed
  derived functor is stably equivalent to $\gamma$ if we forget the
  augmentation to $\bof{\cJ}$. As discussed in connection with
  Corollary~\ref{cor:hocolim-colim-zig-zag}, this is not a stable
  equivalence over $\bof{\cJ}$. We lift the approach of
  Lemma~\ref{lem:hocolim-colim-over-BK} to $\GammaS$ to overcome
  this: Applying the natural map from $\hocolim_{\HcJ(S)}$ to
  $\colim_{\HcJ(S)}$ levelwise defines a level equivalence
  $\gamma'(E\HcJ) \to \bof{\cJ}$ which is different from the level
  equivalence $\gamma'(E\HcJ)\to\gamma'(*) \iso \bof{\cJ}$ induced by
  the map $E\HcJ \to *$. Composition with and base change along these maps
  induces a zig-zag of Quillen equivalences
  \[(\GammaS)_{\sta}/\bof{\cJ} \rightleftarrows
  (\GammaS)_{\sta}/\gamma'(E\HcJ) \leftrightarrows
  (\GammaS)_{\sta}/\gamma'(*),\] and Lemma~\ref{lem:hocolim-colim-over-BK}
  shows that the composite of~\eqref{eq:chain-sta-stable-equiv} with
  this zig-zag has a composed derived functor that is stably
  equivalent to $\gamma$ in $\GammaS/\bof{\cJ}$.
\end{proof}

\subsection{Proof of Proposition~\ref{prop:J-model-pre-stable-equiv}}
The following lemmas will be used to prove that the adjunction
$\Phi\colon(\GammaJS)_{\pre} \rightleftarrows \cC\cS^{\cJ}\colon\Psi$
is a Quillen equivalence with respect to the positive $\cJ$-model
structure on $\cC\cS^{\cJ}$.
\begin{lemma}\label{lem:Phi-Psi-Q-adj}
  The adjunction $\Phi\colon(\GammaJS)_{\pre} \rightleftarrows
  \cC\cS^{\cJ}\colon\Psi$ is a Quillen adjunction.
\end{lemma}
\begin{proof}
  We know from Lemma~\ref{lem:Phi-Psi-level-Q-adj} that $(\Phi,\Psi)$
  is a Quillen adjunction with respect to the level model structure on
  $\GammaJS$. So it is enough to show $\Psi$ sends a positive
  $\cJ$-fibration $f\colon A \to B$ to a fibration in
  $(\GammaJS)_{\pre}$.  By Lemma~\ref{lem:gamma-prime-detects-we} and
  Proposition~\ref{prop:gamma-A-special}, the level fibrant
  replacements of $\Psi(A)$ and $\Psi(B)$ are fibrant in
  $(\GammaJS)_{\pre}$. Since $(\GammaJS)_{\lev}$ is right proper, the
  general criterion~\cite[Proposition 3.4.7]{Hirschhorn_model} applies
  to show that $\Psi(f)$ is a fibration in $(\GammaJS)_{\pre}$.
\end{proof}

\begin{lemma}\label{lem:Psi-preserves-detects-we}
A map $f\colon A \to B$ in $\cC\cS^{\cJ}$ is a $\cJ$-equivalence if and only if
$\Psi(f)$ is a pre-stable equivalence in $\GammaJS$.
\end{lemma}
\begin{proof}
  Lemma~\ref{lem:gamma-prime-detects-we} implies that $\Psi(f)$ is a
  pre-stable equivalence if and only if
  $(\gamma'\Psi)(f)\iso\gamma(f)$ is a pre-stable equivalence in
  $\GammaS$.  By Proposition~\ref{prop:gamma-A-special}, $\gamma(A)$
  and $\gamma(B)$ are special. Hence $\gamma(f)$ is a pre-stable
  equivalence if and only if it is a weak equivalence when evaluated
  at $1^+$. Because $\HcJ(1^+)\iso \cJ$, this holds if and only if $f$
  is a $\cJ$-equivalence.
\end{proof}

The right adjoint $\Psi$ preserves coproducts up to weak equivalence:
\begin{lemma}\label{lem:psi-preserves-sums}
  Let $A$ and $B$ be positive cofibrant commutative $\cJ$-space
  monoids, and let $(\Psi(A))^{\cof}$ and
  $(\Psi(B))^{\cof}$ be cofibrant replacements of $\Psi(A)$
  and $\Psi(B)$ in $(\GammaJS)_{\pre}$. The canonical map
  $(\Psi(A))^{\cof}\coprod(\Psi(B))^{\cof} \to
  \Psi(A\boxtimes B)$ is a pre-stable equivalence.
\end{lemma}
\begin{proof}
  By Lemma~\ref{lem:gamma-prime-detects-we}, it is enough to show that
  the map in question induces a pre-stable equivalence after applying
  $\gamma'\colon \GammaJS\to\GammaS$. Since $\gamma'$ is levelwise
  defined as a homotopy colimit,
  Corollary~\ref{cor:colimits-in-GammaJS} and the fact that
  $(\HcJ(S)(\bld{0}_S,-))_{h\HcJ(S)}$ is contractible imply that
  $\gamma'(\Psi(A)^{\cof}) \wdg \gamma'(\Psi(B)^{\cof}) \to
  \gamma'(\Psi(A)^{\cof}\textstyle\coprod \Psi(B)^{\cof})$ is a level
  equivalence. Hence it suffices to show that $\gamma(A) \wdg
  \gamma(B) \to \gamma(A\boxtimes B)$ is a pre-stable equivalence in
  $\GammaS$. We claim that the last map factors as
\begin{equation}\label{eq:wdg-to-boxtimes-w-factorization}
\gamma(A) \wdg \gamma(B) \to \gamma(A) \times \gamma(B) \to \gamma(A\boxtimes B) 
\end{equation}
where $\gamma(A)\times\gamma(B)$ denotes the cartesian product in
$\GammaS$. (This is not the product in $\GammaS/\bof{\cJ}$!) In each level,
the second map in~\eqref{eq:wdg-to-boxtimes-w-factorization} is defined as 
the composite
\begin{equation}\label{eq:gen-mon-str-map}
  \xymatrix@-1pc{(A_S)_{h\HcJ(S)} \times (B_S)_{h\HcJ(S)} \ar[r]&
    ((-\concat_{\HcJ(S)}-)^*(A\boxtimes B)_S)_{h(\HcJ(S)\times\HcJ(S))} \ar[d] \\ 
    & ((A\boxtimes B)_S)_{h\HcJ(S)}.
  }\end{equation} Here $\concat_{\HcJ(S)}$ is the monoidal product in
$\HcJ(S)$, and the map~\eqref{eq:gen-mon-str-map} is induced by the canonical natural
transformation $A_S \times B_S \to (-\concat_{\HcJ(S)}-)^*(A\boxtimes
B)_S$. We omit the long but straightforward proof that this defines
a map of $\Gamma$-spaces. The first map
in~\eqref{eq:wdg-to-boxtimes-w-factorization} is the canonical map
from the coproduct to the product, and the composite
is indeed the map we are interested in.

The first map in~\eqref{eq:wdg-to-boxtimes-w-factorization} is a
pre-stable equivalence because for all pairs of objects in
$\Ho((\GammaS)_{\pre})$, the map from the coproduct to the product is
an isomorphism. This is a consequence of the equivalence between the
homotopy category of special $\Gamma$-spaces and the homotopy category
of $E_{\infty}$ spaces~\cite[Theorem 1.9]{Mandell_inverse-K-theory}
and the corresponding statement about the homotopy category of
$E_{\infty}$ spaces. The latter follows for example from~\cite[Theorem
1.2]{Sagave-S_diagram} and~\cite[Proposition
2.27]{Sagave-S_group-compl}.

Since $\gamma(A)\times \gamma(B)$ and $\gamma(A\boxtimes B)$ are
special, the second map in~\eqref{eq:wdg-to-boxtimes-w-factorization}
is a pre-stable equivalence if its evaluation at $1^+$ is a weak
equivalence of spaces. This evaluation is isomorphic to the monoidal
structure map $A_{h\cJ} \times B_{h\cJ} \to (A\boxtimes B)_{h\cJ}$ of
$(-)_{h\cJ}$. Since $A$ and $B$ are assumed to be cofibrant, their
underlying $\cJ$-spaces are flat in the sense of~\cite[Section
4.27]{Sagave-S_diagram}. Hence a similar argument as given for the
corresponding statement about $\cI$-spaces~\cite[Lemma
2.25]{Sagave-S_group-compl} shows that the map is a weak equivalence.
\end{proof}

We call a commutative square
\begin{equation}\label{eq:hty-cocartesian-test-square}
\xymatrix@-1pc{U \ar[d] \ar[r] & X \ar[d] \\ V \ar[r] & Y}
\end{equation}
in $(\GammaJS)_{\pre}$ \emph{homotopy cocartesian} if for a
factorization $\xymatrix@1@-1pc{U\, \ar@{>->}[r] & \,V'\,
  \ar@{->>}[r]^{\sim} & \, V}$ of $U \to V$ into a cofibration
followed by an acyclic fibration the induced map $V'\coprod_U X \to Y$
is a pre-stable equivalence. This is well defined since
$(\GammaJS)_{\pre}$ is left proper and left properness implies the
gluing lemma~\cite[II.\S 8]{Goerss-J_simplicial}. Homotopy cocartesian
squares in $\cC\cS^{\cJ}$ are defined similarly.
\begin{lemma}\label{lem:psi-preserves-hty-cart}
  The right adjoint $\Psi$ preserves homotopy cocartesian squares.
\end{lemma}
\begin{proof}
  Let~\eqref{eq:hty-cocartesian-test-square} be a homotopy cocartesian
  test square, let $\cP$ be the category $(b \ot a \to c)$, and let
  $P\colon \cP \to \cC\cS^{\cJ}$ be the diagram $V \ot U \to X$.  In
  view of Lemma~\ref{lem:Psi-preserves-detects-we}, we may assume that
  $\cP$ is a cofibrant diagram. This means that all objects are
  cofibrant and $U\to V$ and $U \to X$ are cofibrations. Let
  $\Psi(P)^{\cof}$ be an cofibrant replacement of $\Psi(P)$ in the
  category of $\cP$-diagrams in $(\GammaJS)_{\pre}$.  Again using
  Lemma~\ref{lem:Psi-preserves-detects-we}, it is enough to show that
  the top map in
\begin{equation}\label{eq:colim-hocolim-square}\xymatrix@-1pc{
\colim_{\cP}\left(\Psi(P)^{\cof}\right) \ar[r] & \Psi(\colim_{\cP}P)\\
\hocolim_{\cP}\left(\Psi(P)^{\cof}\right) \ar[r] \ar[u] & \Psi(\hocolim_{\cP}P)\ar[u]}
\end{equation}
is a weak equivalence. Using the coproducts in the respective
categories and the realization functor given by the diagonal, we may
also form the Bousfield-Kan homotopy colimits of $P$ and
$\Psi(P)^{\cof}$ displayed in the bottom row (compare
Definition~\ref{def:K-equivalences}). The vertical maps
in~\eqref{eq:colim-hocolim-square} are the canonical maps from the
homotopy colimit to the colimit~\cite[Example
18.3.8]{Hirschhorn_model} and make the diagram commutative. Since $P$
and $\Psi(P)^{\cof}$ are cofibrant diagrams, an argument similar
to~\cite[Proposition 18.9.4]{Hirschhorn_model} shows that the vertical
maps are weak equivalences.

It remains to show that the bottom map
in~\eqref{eq:colim-hocolim-square} is a weak equivalence. By
Lemma~\ref{lem:psi-preserves-sums}, it is a weak equivalence in every
degree of the simplicial replacement used to define the homotopy
colimit. Hence Lemma~\ref{lem:pre-stable-realization-lemma} shows the
claim.
\end{proof}
\begin{lemma}\label{lem:pre-stable-realization-lemma}
  If a map of simplicial objects in $\GammaJS$ is a pre-stable
  equivalence in every simplicial degree, then its diagonal is a
  pre-stable equivalence.
\end{lemma}
\begin{proof}
  The functor $\gamma'\colon \GammaJS\to\GammaS$ and the induced
  functor on the simplicial objects in these categories commute up to
  isomorphism with the diagonal. So by
  Lemma~\ref{lem:gamma-prime-detects-we}, it is enough to show that a
  map $X_{\bullet}\to Y_{\bullet}$ of simplicial objects in $\GammaS$
  such that $X_p \to Y_p$ is a pre-stable equivalence gives a
  pre-stable equivalence after taking the diagonal. Viewing
  $X_{\bullet}\to Y_{\bullet}$ as a weak equivalence in the Reedy
  model structure on simplicial objects in $(\GammaS)_{\pre}$, this
  follows from~\cite[Theorem 18.6.6]{Hirschhorn_model} if
  $X_{\bullet}$ and $Y_{\bullet}$ are Reedy cofibrant. The general
  case follows because a Reedy cofibrant replacement $X'_{\bullet} \to
  X_{\bullet}$ is a Reedy acyclic fibration and hence a degreewise
  acyclic fibration. These are degreewise level equivalences in
  $\GammaS$ and hence level equivalence after taking diagonal by the
  realization lemma for bisimplicial sets.
\end{proof}
\begin{lemma}\label{lem:adjunction-unit-on-free-in-deg-one}
The adjunction unit $\id_{\GammaJS} \to \Psi\Phi$ is a pre-stable equivalence
on the free $\GammaJ$-space $\GammaJ((1^+;\bld{s},\sigma),-)$ for every
positive object $(\bld{s},\sigma)$ of $\HcJ(1^+)$. 
\end{lemma}
\begin{proof}
  We write $X$ for $\GammaJ((1^+;\bld{s},\sigma),-)$ and $\bld{k}$ for
  the image of $(\bld{s},\sigma)$ under the isomorphism
  $\HcJ(1^+)\to\cJ$.  By Lemma~\ref{lem:gamma-prime-detects-we}, it is
  enough to show that the map $\gamma'(X) \to \gamma'(\Psi(\Phi(X)))
  \iso \gamma(\Phi(X))$ is a pre-stable equivalence. The map
  $\Gamma^{\op}(1^+,-) \to \gamma'(X)$ sending $\id_{1^+}$ to the
  $0$-simplex in $X(1^+;-)_{h\HcJ(1^+)}$ specified by $\id_{1^+}$ and
  $\id_{\bld{k}}$ is a level equivalence since the
  $\HcJ(T)(T;-)_{h\HcJ(T)}$ are contractible.

  The definition of $\Phi$ implies that there are isomorphisms 
  \[ \Phi(X) \iso \mC F^{\cJ}_{\bld{k}}(*) \iso
  \textstyle\coprod_{p\geq 0} \cJ(\bld{k}^{\concat p}, -)/\Sigma_p.\]
  Let $\cC_{\bld{k}}$ be the category whose objects are pairs
  $([\rho],\bld{l})$ where $\rho \colon \bld{k}^{\concat p} \to
  \bld{l}$ is a morphism in $\cJ$ and $[\rho]$ is its equivalence
  class in $ \cJ(\bld{k}^{\concat p}, \bld{l})/\Sigma_p.$ Morphisms in
  $\cC_{\bld{k}}$ are maps $\varphi\colon \bld{l_1}\to\bld{l_2}$ such
  that $[\rho_2]=[\varphi\rho_1]$. Since $\mC F^{\cJ}_{\bld{k}}(*)$ is
  a $\cJ$-space with values in discrete simplicial sets, it follows
  from the definition in the homotopy colimit that $B\cC_{\bld{k}}
  \iso (\mC F^{\cJ}_{\bld{k}}(*))_{h\cJ}$. Moreover, $\cC_{\bld{k}}$
  is symmetric monoidal with
  \[ ([\rho\colon \bld{k}^{\concat p} \to \bld{l}],\bld{l}) \concat
  ([\rho\colon \bld{k}^{\concat p'} \to \bld{l'}],\bld{l'}) =
  ([\rho\concat\rho' \colon \bld{k}^{\concat p+p'} \to
  \bld{l}\concat\bld{l'}],\bld{l}\concat\bld{l'}).\] Inspecting the
  definitions of Section~\ref{sec:Gamma-spaces-from-monoids}, we see
  that the last isomorphism extends to an isomorphism of
  $\Gamma$-spaces $\bof{\cC_{\bld{k}}} \iso \gamma(\mC
  F^{\cJ}_{\bld{k}}(*))$.

  Now we let $\Sigma=\mathrm{Iso}(\cI)$ be the symmetric monoidal
  category of finite sets and bijections. Sending $\bld{p}$ to
  $([\id_{\bld{k}^{\concat p}}],\bld{k}^{\concat p})$ defines a
  symmetric monoidal functor $\Sigma \to \cC_{\bld{k}}$. It induces a map
  of special $\Gamma$-spaces $\bof{\Sigma} \to \Gamma(\mC
  F^{\cJ}_{\bld{k}}(*))$ which is a level equivalence because
  \[ \bof{\Sigma}(1^+) \to \gamma(\mC F^{\cJ}_{\bld{k}}(*)) (1^+) \iso
  (\mC F^{\cJ}_{\bld{k}}(*))_{h\cJ} \simeq \textstyle\coprod_{p\geq 0}
  B\Sigma_p \] is a weak equivalence. The last weak equivalence uses
  that $\bld{k}$ is positive and follows from a similar argument as in
  the case of $\cI$-spaces~\cite[Example 3.7]{Sagave-S_group-compl}.
 
  Segal's proof of the Barratt-Priddy-Quillen
  theorem~\cite[Proposition 3.5]{Segal_categories} implies that the
  map $\Gamma^{\op}(1^+,-) \to \bof{\Sigma}$ sending $\id_{1^+}$ to the
  component of $\id_{\bld{1}}$ in $(\bof\Sigma)(1^+)$ induces an
  isomorphism in the homotopy category of special $\Gamma$-spaces. So
  \mbox{$\Gamma^{\op}(1^+,-)\to\bof{\Sigma}$} is a pre-stable
  equivalence. Because the component of $\id_{\bld{k}}$ is in the
  image of $\pi_0(\gamma'(X)(1^+)) \to \pi_0(\gamma(\mC
  F^{\cJ}_{\bld{k}}(*))(1^+))$, the above level equivalences shows
  that $\gamma'(X)\to\gamma(\Phi(X))$ is a pre-stable equivalence.
\end{proof}
\begin{lemma}\label{lem:adjunction-unit-on-general-free}
  The adjunction unit $\id_{\GammaJS} \to \Psi\Phi$ is a pre-stable
  equivalence on the $\GammaJ$-space
  $\GammaJ((S;\bld{s},\sigma),-)\tensor K$ for every positive object
  $(\bld{s},\sigma)$ of $\HcJ(S)$ and every finite simplicial set $K$.
\end{lemma}
\begin{proof}
  We write $X$ for the $\GammaJ$-space $\GammaJ((S;\bld{s},\sigma),-)$
  and show first that the map $X \to \Psi\Phi(X)$ is a pre-stable
  equivalence. Setting $X_i=\GammaJ((S;(p_i)_*(\bld{s},\sigma)),-)$,
  the maps $(S;\bld{s},\sigma) \to (1^+;(p_i)_*(\bld{s},\sigma))$
  induce a commutative diagram
\[\xymatrix@-1pc{
  \textstyle\coprod_{i\in \ovl{S}} X_i \ar[r] \ar[d] & X \ar[d] \\
  \textstyle\coprod_{i\in \ovl{S}} (\Psi\Phi(X_i))^{\cof}
  \ar[r] & (\Psi\Phi(X)). }\] The top horizontal map is a pre-stable
equivalence by the definition of the set $P$ used to construct the
localization. Because $\coprod_{i\in\ovl{S}}\Phi(X_i)\to\Phi(X)$ is an
isomorphism, Lemma~\ref{lem:psi-preserves-sums} implies that the
bottom horizontal map is a pre-stable equivalence. By
Lemma~\ref{lem:adjunction-unit-on-free-in-deg-one}, the left map is a
coproduct of pre-stable equivalences between cofibrant objects. It
follows that $X\to(\Psi\Phi(X))$ is a pre-stable equivalence.

To show that $X\tensor K\to(\Psi\Phi(X\tensor K))$ is a pre-stable equivalence
we first note that for $\GammaJ$-space $Y$ there is an isomorphism 
\begin{equation}\label{eq:tensor-as-realization-of-lev-coprod}
Y \to K = Y \sm K_+ \iso \left|[m]\mapsto \textstyle\coprod_{K_m}Y\right|.
\end{equation}
Choosing a cofibrant replacement $\xymatrix@1{X \ar@{>->}[r] &
  (\Psi\Phi(X))^{\cof} \ar[r]^-{\sim} & \Psi\Phi(X)}$ of
$\Psi\Phi(X)$, the universal property of the coproduct induces a map
\begin{equation}\label{eq:tensor-vs-adjunction-unit}
  \xymatrix{\left|[m]\mapsto \textstyle\coprod_{K_m}(\Psi\Phi(X))^{\cof}\right| \ar[r] & 
    \left|[m]\mapsto (\Psi\Phi(\textstyle\coprod_{K_m} X))^{\cof}\right|.
  }\end{equation}
Under the isomorphism~\eqref{eq:tensor-as-realization-of-lev-coprod},
this becomes  $(\Psi\Phi(X))^{\cof} \tensor K \to
(\Psi\Phi(X\tensor K))$. Applying Lemma~\ref{lem:psi-preserves-sums} and
Lemma~\ref{lem:pre-stable-realization-lemma}
to~\eqref{eq:tensor-vs-adjunction-unit} shows that the last map is a
pre-stable equivalence. Precomposing it with $(X \to
(\Psi\Phi(X))^{\cof})\tensor K$ gives the adjunction unit on
$X\tensor K$. Since the pre-stable model structure is simplicial and
we already showed that $X \to \Psi\Phi(X)$ is a pre-stable
equivalence, the claim follows by the 2-out-of-3 property for
pre-stable equivalences.
\end{proof}

\begin{proof}[Proof of
  Proposition~\ref{prop:J-model-pre-stable-equiv}]
  We use the criterion of~\cite[Corollary 1.3.16 (c)]{Hovey_model} to
  see that $(\Phi,\Psi)$ is a Quillen equivalence. In view of
  Lemma~\ref{lem:Phi-Psi-Q-adj} and
  Lemma~\ref{lem:Psi-preserves-detects-we}, it remains to show that the
  adjunction unit $X \to \Psi\Phi(X)$ is a pre-stable equivalence
  for all cofibrant objects $X$ in $(\GammaJS)_{\pre}$.

  Lemma~\ref{lem:adjunction-unit-on-general-free} and
  Corollary~\ref{cor:tensor-and-free} imply that the adjunction
  unit is a pre-stable equivalence on the domains and codomains of
  the generating cofibrations. The adjunction unit is also a
  pre-stable equivalence on the initial object in $\GammaJS$ because
  $\Phi$ maps it to the initial commutative $\cJ$-space monoid
  $F_{\bld{0}}^{\cJ}(*)$ and $(F_{\bld{0}}^{\cJ}(*))_{h\cJ}$ is
  contractible.

  Now let $\xymatrix@1@-1pc{V & \ar@{>->}[l] \,\,U\, \ar[r] & \, X}$
  be a diagram of cofibrant objects in $(\GammaJS)_{\pre}$ such that
  the adjunction unit is a pre-stable equivalence on $U, V,$ and
  $X$. The fact that the left Quillen functor $\Phi$ preserves
  homotopy cocartesian squares of cofibrant objects,
  Lemma~\ref{lem:psi-preserves-hty-cart}, and the gluing lemma in
  $(\GammaJS)_{\pre}$ imply that the adjunction unit on the pushout of
  the diagram is a pre-stable equivalence.

  Next let $ \dots \to X_i \to X_{i+1} \to \dots $ be a sequence of
  cofibrations between cofibrant objects in $(\GammaJS)_{\pre}$ such
  that the unit is a pre-stable equivalence on every $X_i$. We claim
  that it is a pre-stable equivalence on $ X = \colim X_i$. The
  forgetful functor $\cC\cS^{\cJ} \to \cS^{\cJ}$ preserves filtered
  colimits. Hence the composite $\Psi\Phi$ commutes with sequential
  colimits, and it is enough to show that the map $\colim X_i \to
  \colim \Psi(\Phi(X_i))$ is a pre-stable equivalence. Since the maps
  between the $X_i$  are cofibrations, the map $\hocolim X_i \to \colim
  X_i$ is a level equivalence. Combining~\cite[Proposition
  7.1(v)]{Sagave-S_diagram} and~\cite[Proposition
  12.7]{Sagave-S_diagram} with the fact that filtered colimits commute
  with products implies that $\hocolim \Psi(\Phi(X_i)) \to \colim
  \Psi(\Phi(X_i))$ is a level equivalence. Together with the homotopy
  invariance of the homotopy colimit this shows that $X \to
  \Psi(\Phi(X))$ is a pre-stable equivalence.

  Every cofibrant object $X$ in $(\GammaJS)_{\pre}$ is the retract of
  the colimit of a sequence of cofibrations starting at the initial
  object in which all maps are cobase changes of the generating
  cofibrations. So the above implies that the adjunction unit is a
  pre-stable equivalence on all cofibrant objects.
\end{proof}

\end{document}